\documentclass[11pt,reqno]{amsart}

\usepackage[margin=1.1in]{geometry}
\usepackage[english]{babel}
\usepackage{amsmath,amssymb,amsfonts,amsthm,mathtools}
\usepackage{dsfont}
\usepackage{xcolor}
\usepackage{hyperref}
\usepackage{tikz-cd}

\hypersetup{
  colorlinks = true,
  urlcolor = blue,
  linkcolor = blue,
  citecolor = black
}

\numberwithin{equation}{section}
\theoremstyle{plain}
\newtheorem{thmctr}{}[section]
\newtheorem{lemma}[thmctr]{Lemma}
\newtheorem{theorem}[thmctr]{Theorem}
\newtheorem{proposition}[thmctr]{Proposition}
\newtheorem{corollary}[thmctr]{Corollary}
\theoremstyle{definition}
\newtheorem{definition}[thmctr]{Definition}
\newtheorem{example}[thmctr]{Example}
\theoremstyle{remark}
\newtheorem{remark}[thmctr]{Remark}
\newtheorem{question}[thmctr]{Question}

\newcommand{\kk}{\Bbbk}
\newcommand{\unit}{\mathds{1}}

\newcommand{\B}{\mathcal{B}}
\newcommand{\C}{\mathcal{C}}
\newcommand{\D}{\mathcal{D}}
\newcommand{\E}{\mathcal{E}}
\newcommand{\F}{\mathcal{F}}
\newcommand{\M}{\mathcal{M}}
\newcommand{\T}{\mathcal{T}}
\newcommand{\X}{\mathcal{X}}
\newcommand{\Z}{\mathcal{Z}}

\newcommand{\loc}{\mathrm{loc}}
\newcommand{\Ind}{\mathrm{Ind}}
\newcommand{\Res}{\mathrm{Res}}
\newcommand{\Hom}{\mathrm{Hom}}
\newcommand{\ra}{\mathrm{ra}}
\newcommand{\rev}{\mathrm{rev}}

\newcommand{\im}{\mathrm{Im}}
\newcommand{\FP}{\mathrm{d}}
\newcommand{\Vect}{\mathrm{vect}}
\newcommand{\id}{\mathrm{id}}
\newcommand{\Fun}{\mathrm{Fun}}
\newcommand{\Ker}{\mathrm{Ker}}

\newcommand{\Lalg}[1]{\mathfrak{L}_{\mathrm{alg}}(#1)}
\newcommand{\Lten}[1]{\mathfrak{L}_{\otimes}(#1)}

\newcommand{\oB}{\overline{\B}}
\newcommand{\ozeta}{\overline{\zeta}}

\newcommand{\tF}{\widetilde{F}}

\newcommand{\bA}{\mathbb{A}}

\newcommand{\I}{\mathrm{I}}
\newcommand{\U}{\mathrm{U}}
\newcommand{\R}{\mathrm{R}}
\newcommand{\Q}{\mathrm{Q}}

\title[Subalgebras and tensor subcategories]{A Galois connection between subalgebras and tensor subcategories}
\author{Harshit Yadav}
\address{Max-Planck-Institut f\"ur Mathematik, Vivatsgasse 7, 53111 Bonn, Germany}
\email{yadav@mpim-bonn.mpg.de}

\author{Kenichi Shimizu}
\address{Department of Mathematical Sciences, Shibaura Institute of Technology, 307 Fukasaku, Minuma-ku, Saitama-shi, Saitama 337-8570, Japan}
\email{kshimizu@shibaura-it.ac.jp}

\date{\today}

\usepackage[normalem]{ulem}

\begin{document}

\begin{abstract}
Let $\B$ be a braided finite tensor category and let $A$ be a simple
commutative algebra in $\B$. We construct an order-reversing Galois connection
between subalgebras of $A$ and tensor subcategories of $\B_A$. Let $\B'$ denote
the M\"uger center and set $A':=A\cap\B'$. The closure operators are
$B\mapsto\langle B,A'\rangle_{\mathrm{alg}}$ and
$\E\mapsto\langle\E,\B_A^{\loc}\rangle_{\otimes}$.
Thus the closed subalgebras are those containing $A'$, while the closed
tensor subcategories are those containing $\B_A^{\loc}$; equivalently, the
fixed-point intervals $[A',A]_{\mathrm{alg}}$ and
$[\B_A^{\loc},\B_A]_{\otimes}$ are anti-isomorphic as lattices.

For a finite tensor category $\C$, the canonical algebra in $\Z(\C)$ yields an
anti-isomorphism between its subalgebras and tensor subcategories of $\C$.
When $\B$ is nondegenerate, Frobenius extensions correspond to unimodular
tensor subcategories. For a Hopf algebra in $\B$, the correspondence
specializes to an order-preserving bijection between Hopf ideals and normal
left coideal subalgebras.
\end{abstract}

\maketitle
 
\section{Introduction}

Let $\C$ be a finite tensor category. Its tensor subcategories form the lattice
$\Lten{\C}$, with intersection as meet and the generated tensor subcategory as
join. This lattice is a basic structural invariant of $\C$, but describing it
in full is generally difficult. A more local version of the problem is to fix
a tensor subcategory $\D\subseteq\C$ and determine all tensor subcategories
lying between $\D$ and $\C$. In lattice-theoretic terms, this amounts to
describing the interval $[\D,\C]_{\otimes}$.

Central tensor subcategories provide especially natural lower endpoints for
such intervals. Here centrality means that the inclusion
$\D\hookrightarrow\C$ admits a braided lift to the Drinfeld center $\Z(\C)$.
The resulting half-braidings give $\D$ additional compatibility with every
object of $\C$, and one expects this compatibility to constrain the position of
$\D$ inside $\Lten{\C}$. Our first result makes this constraint precise.
Recall that an element $a$ of a lattice is left-modular if
\[
  x\vee(a\wedge y)=(x\vee a)\wedge y
  \qquad\text{whenever }x\leq y.
\]
We prove that every central tensor subcategory $\D\subseteq\C$ is left-modular
in $\Lten{\C}$. In particular, every tensor subcategory of a braided finite
tensor category is central, so its lattice of tensor subcategories is modular.
In the fusion setting, this modularity result follows from
\cite[Lemma~5.6]{drinfeld2010braided}.

We now apply this viewpoint to module categories. Let $\B$ be a braided finite
tensor category, and let $A$ be a simple commutative algebra in $\B$. We write
$\B_A$ for the category of right $A$-modules and $\B_A^\loc$ for its full
subcategory of local modules. Since $A$ is simple, it is exact
\cite{coulembier2025simple}, and hence $\B_A$ and $\B_A^\loc$ are finite tensor
categories \cite{shimizu2026commutative}. Moreover, $\B_A^\loc$ is braided and
the inclusion $\B_A^\loc\hookrightarrow\B_A$ is central
\cite{pareigis1995braiding}. Therefore $\B_A^\loc$ is left-modular in
$\Lten{\B_A}$.

The interval $[\B_A^\loc,\B_A]_{\otimes}$ measures the possible
tensor-theoretic enlargements of the braided category of local modules inside
the category of all $A$-modules. Subalgebras enter naturally into this problem:
an inclusion $B\subseteq A$ gives restriction and induction functors between
the corresponding module categories. This leads us to ask whether intermediate
tensor subcategories of $\B_A$ can be described in terms of subalgebras of
$A$.

There is a direct semisimple precedent for this problem. Let $\B$ be a
nondegenerate braided fusion category and let $A$ be a connected \'{e}tale
algebra in $\B$. Davydov--M\"uger--Nikshych--Ostrik showed that the lattice of
subalgebras of $A$ is anti-isomorphic to the interval of fusion subcategories
of $\B_A$ containing $\B_A^\loc$
\cite[Theorem~4.10 and Remark~4.12]{davydov2013witt}.

The aim of the present paper is to determine what remains true for a simple
commutative algebra in an arbitrary braided finite tensor category, without
assuming semisimplicity or nondegeneracy. In this generality, the natural
order-reversing assignments need not be mutually inverse on the two full
lattices. A Galois connection is therefore the appropriate replacement: its
composites are closure operators, and its closed elements still form
anti-isomorphic lattices. We construct this Galois connection and compute both
closure operators explicitly.

This setup also contains the problem of describing all tensor subcategories of
an arbitrary finite tensor category. Indeed, let $\C$ be a finite tensor
category, let $\I_{\C}:\C\to\Z(\C)$ be the right adjoint of the forgetful
functor, and let $\bA=\I_{\C}(\unit)$ be the canonical algebra in $\Z(\C)$.
Then
\[
  \Z(\C)_{\bA}\simeq\C,
  \qquad
  \Z(\C)_{\bA}^{\loc}\simeq\Vect.
\]
Consequently, the interval
$[\Z(\C)_{\bA}^{\loc},\Z(\C)_{\bA}]_{\otimes}$ is identified with the full
lattice $\Lten{\C}$.

\subsection*{The Galois correspondence and its consequences}
We first formulate the correspondence on the two full lattices. We then
compute the closure operators and identify the intervals on which the two
maps become mutually inverse.

Write $\B'$ for the M\"uger center of $\B$ and set $A':=A\cap\B'$. 
With this notation, consider the two full lattices and their distinguished
intervals:
\begin{equation*}
\begin{aligned}
\Lalg{A}
  &= \{\,\text{algebras }B\mid B\xhookrightarrow{i_B}A\,\},\\
[A',A]_{\mathrm{alg}}
  &= \{\,B\in\Lalg{A}\mid A'\hookrightarrow B\,\},\\[2pt]
\Lten{\B_A}
  &= \{\,\text{tensor subcategories }\E\mid
       \E\xhookrightarrow{\iota_\E}\B_A\,\},\\
[\B_A^\loc,\B_A]_{\otimes}
  &= \{\,\E\in\Lten{\B_A}\mid
       \B_A^\loc\subseteq\E\,\}.
\end{aligned}
\end{equation*}
These lattices and intervals are ordered by inclusion. For a subalgebra
$B\subseteq A$,
induction of local $B$-modules defines a tensor subcategory $\beta(B)$ of
$\B_A$. In the other direction, we use the right adjoint of a central functor $\B \to \Z(\E;\B_A)$ (where the target is the relative Drinfeld center) to define a subalgebra $\alpha(\E)\subseteq A$ for each tensor
subcategory $\E\subseteq\B_A$. We show that the resulting order-reversing maps
$\alpha:\Lten{\B_A}\to\Lalg{A}$ and
$\beta:\Lalg{A}\to\Lten{\B_A}$ satisfy the Galois relation:
\[
  \E\subseteq\beta(B)
  \quad\Longleftrightarrow\quad
  B\subseteq\alpha(\E).
\]
Thus $\alpha(\E)$ is the largest subalgebra $B\subseteq A$ for which
$\E\subseteq\beta(B)$. Dually, $\beta(B)$ is the largest tensor subcategory
$\E\subseteq\B_A$ for which $B\subseteq\alpha(\E)$. 
The corresponding closure formulas are
\[
  \alpha\beta(B)=\langle B,A'\rangle_{\mathrm{alg}},
  \qquad
  \beta\alpha(\E)=\langle\E,\B_A^{\loc}\rangle_\otimes.
\]
Thus a subalgebra is
closed precisely when it contains $A'$, and a tensor subcategory is closed
precisely when it contains $\B_A^{\loc}$. It follows that $\alpha$ and $\beta$
restrict to mutually inverse lattice anti-isomorphisms
\begin{equation}\label{eq:intro-lattice-anti-isomorphism}
  [\B_A^{\loc},\B_A]_{\otimes}
  \underset{\beta}{\overset{\alpha}{\rightleftarrows}}
  [A',A]_{\mathrm{alg}}.
\end{equation}
These formulas identify $A'$ and $\B_A^\loc$ as precisely the two
obstructions to obtaining a correspondence on the full lattices.

The nondegenerate and symmetric cases exhibit the two extreme behaviors of
the closure operators. If $\B$ is nondegenerate, then
$A'\cong\unit$. Hence every subalgebra of $A$ is closed, and
\eqref{eq:intro-lattice-anti-isomorphism} becomes an anti-isomorphism between
all subalgebras of $A$ and all tensor subcategories lying between $\B_A^{\loc}$ and $\B_A$. If $\B$ is
symmetric, then $A'=A$ and $\B_A^{\loc}=\B_A$. Both Galois maps are constant,
so they do not distinguish proper subalgebras or proper tensor subcategories.
Thus, by \eqref{eq:intro-lattice-anti-isomorphism}, the maps are mutually inverse on the two full lattices precisely when
$A'\cong\unit$ and $\B_A^{\loc}\simeq\Vect$. 

The correspondence also detects Frobenius extensions. Let $B\subseteq A$ be a
subalgebra such that $\B_B^\loc$ is unimodular. We prove that the extension
$B\subseteq A$ is Frobenius if and only if the corresponding tensor
subcategory $\beta(B)$ is unimodular. In particular, this criterion applies
to every subalgebra $B\subseteq A$ when $\B$ is nondegenerate.

\subsection*{Applications}

We first return to the canonical algebra
$\bA=\I_{\C}(\unit)\in\Z(\C)$ introduced above. This algebra is simple and
commutative, and $\Z(\C)$ is nondegenerate. The canonical equivalences above
and the main theorem therefore give
\begin{equation*}
  \Lalg{\bA}^{\mathrm{op}}
  \cong
  \Lten{\C}.
\end{equation*}
We then apply the same construction to the adjoint algebra $\bA_\M$
associated with an indecomposable exact module category, thereby describing
the tensor subcategories of the corresponding dual tensor category. In the
two settings, Frobenius extensions of $\bA$ and $\bA_\M$, respectively,
correspond exactly to unimodular tensor subcategories.

Our main application concerns Hopf algebras in braided finite tensor
categories.
Let $H$ be a Hopf algebra in $\B$. A Hopf ideal $J$ determines the coinvariant
subalgebra $H^{\operatorname{co}(H/J)}$, while a normal left coideal subalgebra
$K$ determines the Hopf ideal $HK^+$, where
$K^+=\Ker(\varepsilon_H|_K)$. We prove that these two constructions are
mutually inverse.

The categorical bridge is the adjoint algebra $\mathbf{A}$ of $H$ in the
Yetter--Drinfeld category ${}^H_H\mathcal{YD}(\B)$. Its subalgebras in
this category are precisely the normal left coideal subalgebras of $H$. On the
module side, Hopf ideals of $H$ correspond to the interval
$[\B,{}_H\B]_{\otimes}$. Under these identifications, our main correspondence
gives mutually inverse order-preserving bijections
\begin{equation*}
\begin{aligned}
  \{\text{Hopf ideals $J$ of }H\}
  &\longleftrightarrow
  \{\text{normal left coideal subalgebras $K$ of }H\},\\
  J&\longmapsto H^{\operatorname{co}(H/J)},
  \qquad
  K\longmapsto HK^+.
\end{aligned}
\end{equation*}
Thus the abstract Galois correspondence becomes an explicit correspondence
between quotient data and coinvariant data for $H$. This generalizes a known correspondence for finite-dimensional Hopf algebras to the braided setting.

\subsection*{Organization}
Section~\ref{sec:preliminaries} fixes the categorical and lattice-theoretic
preliminaries used throughout the paper. Section~\ref{sec:modular} proves that
central tensor subcategories are left-modular. In
Section~\ref{sec:galois-connection}, we construct the maps $\alpha$ and
$\beta$, establish the Galois relation, compute the closure operators, and
derive the Frobenius and canonical-algebra applications. Finally,
Section~\ref{sec:hopf-ideals} realizes the correspondence through the adjoint
algebra of a Hopf algebra and proves the bijection between Hopf ideals and
normal left coideal subalgebras. 

\subsection*{Acknowledgements}
HY was partly supported by a start-up grant from the University of Alberta and
an NSERC Discovery Grant. KS was supported by JSPS KAKENHI Grant Number JP24K06676. We used ChatGPT to explore proof strategies for a few parts of Sections~\ref{sec:modular} and \ref{sec:galois-connection}. All mathematical statements and proofs were independently checked by the authors, who take sole responsibility for the contents of the paper.


\section{Preliminaries}\label{sec:preliminaries}
Unless otherwise stated, the base field $\kk$ is algebraically closed.

Given a functor $F:\C\to\D$, we write $F(\C)$ for the essential image of $F$. We write $F^{\ra}$ for the right adjoint of $F$, if it exists. The unit and counit of the adjunction $F\dashv F^\ra$ are denoted as $\eta$ and $\varepsilon$, respectively. The adjunction yields isomorphisms
\begin{equation}\label{eq:zeta}
\begin{aligned}
   \Hom_{\C}(X, F^\ra(Y)) \xrightarrow{\zeta_F} \Hom_{\D}(F(X),Y) , \qquad
   f \longmapsto \left[ F(X) \xrightarrow{F(f)} FF^\ra(Y) \xrightarrow{\varepsilon_Y} Y  \right].
\end{aligned}
\end{equation}
The inverse of $\zeta_F$ is defined using $\eta$ and denoted as $\ozeta_F$.

When $F:\C\to\D$ is an exact functor between abelian categories, we write $\im(F)$ for the full subcategory of $\D$ consisting of subquotients of objects in $F(\C)$. The Deligne product of two locally finite $\kk$-linear abelian categories $\C$ and $\D$ is denoted by $\C \boxtimes \D$ \cite[\S1.11]{etingof2016tensor}.

We denote monoidal categories as $(\C,\otimes,\unit)$ by suppressing the associators and unitors. We write $\C^{\rev}$ for the monoidal category obtained from $\C$ by reversing the tensor product. 


\subsection{Finite tensor categories}
We use the standard terminology of finite tensor categories as in \cite{etingof2016tensor}. A \emph{finite tensor category} \cite{etingof2004finite} is a rigid monoidal category $\C$ such that $\C$ is a finite abelian category, the tensor product is $\kk$-linear in each variable and the unit object is simple. A \emph{tensor functor} between finite tensor categories is a $\kk$-linear strong monoidal functor which is also exact. Tensor functors are faithful \cite[Corollaire 2.10 (i)]{deligne1990categories} and admit right adjoints.

A \emph{tensor subcategory} of a finite tensor category $\C$ is a full subcategory $\D$ of $\C$ containing $\unit_\C$ such that $\D$ is closed under taking subobjects, quotients, finite direct sums, tensor products and duals. A tensor subcategory is itself a finite tensor category and the inclusion functor is a tensor functor. Given a tensor functor $F:\C\to\D$ between finite tensor categories, $\im(F)$ is a tensor subcategory of $\D$.


\subsubsection{Frobenius-Perron dimension}
We refer the reader to \cite[\S4]{etingof2016tensor} for the definition of Frobenius-Perron dimension of objects and categories. We denote the Frobenius-Perron dimension of an object $X$ in a finite tensor category $\C$ by $\FP_{\C}(X)$ and that of $\C$ by $\FP(\C)$. 

We call a tensor functor $F$ \emph{surjective} if $\im(F)=\D$ and we call it \emph{injective} if $F$ is fully faithful.  We note some well-known facts about tensor functors that we will frequently use.
\begin{lemma}\label{lem:tensor-functor-facts}
  Let $F:\C\to\D$ be a tensor functor between finite tensor categories.
  \begin{enumerate}
    \item $F$ is surjective if and only if $F^\ra$ is faithful. 
    \item If $F$ is surjective, then $F^{\ra}$ is exact.
    \item If $F$ is injective, then $\im(F)=F(\C)$.
  \end{enumerate}
\end{lemma}
\begin{proof}
Part (1) is \cite[Lemma~3.1 and Remark~3.3]{bruguieres2011exact}. Part (2) follows from \cite[Theorem~2.5]{etingof2004finite} and \cite[Lemma 2.1]{shimizu2017relative}. For (3), consider the corestriction $\C\to\im(F)$. Since $F$ is injective, \cite[Corollary~6.3.5(i)]{etingof2016tensor} gives
$\FP(\im(F))=\FP(\C)$. The corestriction is surjective, and hence it is
an equivalence by \cite[Proposition~6.3.4]{etingof2016tensor}. Therefore
every object of $\im(F)$ is isomorphic to $F(X)$ for some $X\in\C$.
\end{proof}

We record another fact about Frobenius-Perron dimension that we will use in this paper. 

\begin{lemma}\label{lem:FPdim}
  Let $F_1:\C\to\D$ and $F_2:\D\to\E$ be tensor functors between finite tensor categories with $F_1$ injective and $F_2$ surjective. Denote $F=F_2\circ F_1$. If $F_2^\ra(\unit_{\E})$ belongs to $\im(F_1)$, then 
  \begin{equation*}
    \FP_\C(F^\ra(\unit_\E)) = \frac{\FP(\D)}{\FP(\E)}.
  \end{equation*}
\end{lemma}
\begin{proof}
  As $F_1$ is injective, by Lemma~\ref{lem:tensor-functor-facts} there exists an object $X$ in $\C$ such that $F_1(X)\cong F_2^\ra(\unit_\E)$. Thus,  
  \[\FP_{\C}(X) = \FP_{\D}(F_1(X)) = \FP_{\D}(F_2^\ra(\unit_\E)) = \frac{\FP(\D)}{\FP(\E)}.\]
  The first equality holds because tensor functors preserve Frobenius-Perron
  dimension \cite[Proposition~4.5.7]{etingof2016tensor}. The last one follows
  from \cite[Lemma~6.2.4]{etingof2016tensor}, since $F_2$ is surjective. On the
  other hand, we have
  \begin{align*}
    \FP_\C(F^\ra(\unit_\E)) = \FP_\C(F_1^\ra(F_2^\ra(\unit_\E))) = \FP_\C(F_1^\ra(F_1(X))) = \FP_\C(X),
  \end{align*}  
  where the last equality holds because $F_1$ is fully faithful, hence the unit $\eta_X : X \to F_1^\ra(F_1(X))$ of $F_1\dashv F_1^\ra$ is an isomorphism. Combining with the previous equation, this proves the claim. 
\end{proof}


\subsubsection{Relative Drinfeld centers}\label{ssec:relative-drinfeld-centers}
Let $\C$ be a finite tensor category and $\D\subseteq \C$ a tensor subcategory. The \emph{relative Drinfeld center} $\Z(\D;\C)$ of $\C$ over $\D$ is the category whose objects are pairs $(V,\sigma)$ consisting of an object $V\in\C$ and a \emph{half-braiding} with the objects of $\D$, that is, a natural isomorphism $\sigma=\{\sigma_X:V\otimes X \to X\otimes V\}_{X\in\D}$ satisfying the hexagon axioms. Morphisms in $\Z(\D;\C)$ are morphisms in $\C$ that are compatible with the natural isomorphisms $\sigma$.

The Drinfeld center $\Z(\C):=\Z(\C;\C)$ of $\C$ is the relative Drinfeld center of $\C$ over itself. Restricting the half-braidings to $\D$ defines a surjective tensor functor \cite[Lemma~4.6]{shimizu2019non}
\begin{equation}\label{eq:ZC-to-ZDC}
  \R_\D: \Z(\C) \to \Z(\D;\C), \quad (V,\sigma)\mapsto (V,\sigma|_{\D}).
\end{equation}
Moreover, there is a forgetful functor 
\begin{equation*}
  \mathrm{U}_\D: \Z(\D;\C) \to \C, \quad (V,\sigma)\mapsto V
\end{equation*}
which is a surjective tensor functor. We denote the right adjoint of $\U_\D$ by $\I_\D$. Being a right adjoint of a strong monoidal functor, the functor $\I_\D$ is lax monoidal. We will denote $\U:=\U_{\C}$ and $\I:=\I_{\C}$.


\subsubsection{Braided finite tensor categories}
If $\B$ is a braided monoidal category, we denote by $\oB$ the same underlying tensor category with the reversed braiding $\overline{c}_{X,Y}:=c_{Y,X}^{-1}$. A typical example of a braided finite tensor category is the Drinfeld center $\Z(\C)$ of a finite tensor category $\C$.

Let $\B$ be a braided monoidal category and $\C$ a monoidal category. A monoidal functor $F:\B\to\C$ is called \emph{central} if there exists a braided monoidal functor $\widetilde{F}:\B \to\Z(\C)$ such that $F = \mathrm{U}_\C\circ \widetilde{F}$. In this case, for $X\in\B$, we denote the half-braidings of $F(X)$ with arbitrary $V\in\C$ by $\sigma_{X,V}: F(X) \otimes V \to V\otimes F(X)$ and these are natural in $X$ and $V$.

For a tensor subcategory $\D\subseteq\B$, we write $\D'$ for its
M\"uger centralizer in $\B$. In particular,
$\B'$ is the M\"uger center of $\B$.
We say that $\B$ is \emph{nondegenerate} if $\B'\simeq \Vect$. 
We say that $\B$ is \emph{factorizable} if the functor $G:\B\boxtimes \oB\to\Z(\B)$ 
defined by $G(X\boxtimes Y) = (X\otimes Y, c_{X,-}\circ c_{-,Y}^{-1})$ is an equivalence. 
It is known that a braided finite tensor category is factorizable if and 
only if it is nondegenerate \cite{shimizu2019non}.


\subsection{Algebras in monoidal categories}
In this subsection, we assume that all monoidal categories under
consideration have coequalizers and that tensor products preserve them
in each variable.

Let $(A,m_A,u_A)$ be an algebra in a monoidal category $\C$. By an algebra map, we always mean a morphism that preserves multiplication and unit. 
We define a \emph{subalgebra} of $A$ to be an algebra $B$ equipped with a monic algebra map $B \to A$.


\subsubsection{Modules and bimodules}

A right $A$-module is a pair $(M,a_M^r)$ consisting of an object $M\in\C$
and a morphism $a_M^r:M\otimes A\to M$ satisfying the usual axioms.
Similarly, a left $A$-module is a pair $(N,a_N^l)$. The relative tensor
product of a right $A$-module $(M,a_M^r)$ and a left $A$-module $(N,a_N^l)$
is defined as the coequalizer of
$a_M^r\otimes\id_N$ and $\id_M\otimes a_N^l$ from
$M\otimes A\otimes N$ to $M\otimes N$.
We denote by $\pi: M\otimes N \to M\otimes_A N$ the epimorphism onto the relative tensor product. We denote by $\C_A$ and ${}_A \C_A$ the categories of right $A$-modules and $A$-bimodules in $\C$, respectively. 

If $\B$ is a braided monoidal category and $A$ is an algebra in $\B$, we call $A$ commutative if $m\circ c_{A,A} = m$.

Suppose that $\C$ is monoidal. We call a pair $(A,\sigma)$ consisting of an algebra $A$ in $\C$ and a half-braiding $\sigma$, \emph{central commutative}, if  $(A,\sigma)$ is a commutative algebra in $\Z(\C)$. Define 
\[ \C_A^{\sigma} = \{(M,a_M^l,a_M^r) \in {}_A \C_A \mid a_M^l = a_M^r\circ \sigma_M \}. \]


\subsubsection{Tensor categories from algebras}
Suppose that $\C$ is a monoidal category with coequalizers such that the tensor product preserves coequalizers in each variable. Then, the category ${}_A \C_A$ is a monoidal category with the relative tensor product $\otimes_A$ over $A$ and the unit object $A$. For a central commutative algebra $(A,\sigma)$ in $\C$, the category $\C_A^{\sigma}$ is a monoidal subcategory of ${}_A \C_A$. 

Let $\B$ be a braided monoidal category and $A$ a commutative algebra in $\B$. Then $\B_A  := \B_A^{c_{A,-}}$ is a monoidal category with the tensor product $\otimes_A$ over $A$ and the unit object $A$. We define
\[ \B_A^{\loc} = \{(M,a_M^r)\in\B_A \mid a_M^r = a_M^r \circ c_{A,M} \circ c_{M,A} \} .\]
Pareigis proved that $\B_A^{\loc}$ is a monoidal subcategory of $\B_A$ that is braided with the braiding $c^A$ induced from that of $\B$ \cite{pareigis1995braiding}. Moreover, he proved that the braiding $c^A_{M,N}$ is well-defined for $M \in \B_A ,N \in \B_A^\loc$, which yields 
\begin{equation*}
 \text{a fully faithful tensor functor } \B_A \to \Z(\B_A^\loc;\B_A).
\end{equation*}
Equivalently, the inclusion $\B_A^{\loc}\hookrightarrow\B_A$ is central.


\subsubsection{Algebras in finite tensor categories}
In this subsection, the categories are finite tensor categories. 
By an \emph{ideal} of an algebra $A$, we mean a subobject $I\subseteq A$ such that the multiplication maps $A\otimes I\to A$ and $I\otimes A\to A$ factor through $I$. We say that an algebra is \emph{simple} if it admits only two ideals: $0$ and $A$. 

\begin{theorem}\label{thm:central-simple-module-category}
  Let $\C$ be a finite tensor category and $(A,\sigma)$ a central commutative algebra in $\C$. Suppose that $A$ is simple in $\C$. Then, the category $\C_A^{\sigma}$ is a finite tensor category and the free module functor $F_A:\C\to \C_A^{\sigma}$ is a surjective tensor functor.
\end{theorem}
\begin{proof}
Since $A$ is simple, it is exact by
\cite[Theorem~7.1]{coulembier2025simple}. Hence $\C_A^\sigma$ is a finite tensor category by \cite[Theorem~5.4]{shimizu2026commutative}.  For $M\in\C_A^\sigma$, the action map $M\otimes A\to M$ is an
epimorphism, since it has a section induced by the unit of $A$. Thus every
object of $\C_A^\sigma$ is a quotient of a free module. Hence $F_A$ is
surjective.
\end{proof}

Also note that if $\B$ is braided and $A$ is a simple commutative algebra
in $\B$, then $\B_A$ is a finite tensor category and $\B_A^\loc$ is a braided finite tensor category \cite[Theorem B]{shimizu2026commutative}.


\subsubsection{Central lift of free functor}
Let $\B$ be a braided finite tensor category and $A$ a commutative simple algebra
in $\B$. For $X\in\B$ and $M\in\B_A$, define
\[
\sigma^-_{X,M}: (X\otimes A)\otimes_A M
\cong X\otimes M
\xrightarrow{c_{M,X}^{-1}} M\otimes X
\cong M\otimes_A(X\otimes A).
\]
This is the inverse-braiding central structure on the free-module functor
associated with the bimodule structures in
\cite[\S8.8]{etingof2016tensor} (see also \cite[\S2.2]{shimizu2026transparent}).
It defines a braided tensor functor
\[
  F_A^-:\oB\longrightarrow\Z(\B_A).
\]
When $F_A^-$ is regarded as a tensor
functor from $\B$, using the common underlying tensor category of
$\B$ and $\oB$, it is not braided in general. In the rest of the paper, we use the notation $\tF_A = F_A^-$. Since $\tF_A:\B\to \Z(\B_A)$ is a tensor functor, it is faithful.


\subsection{Algebra maps and monoidal adjunctions}\label{ssec:alg-mon-adj}
If $F:\C\to\D$ is a lax monoidal functor, and $f:A\to B$ is an algebra map in $\C$, then $F(f):F(A)\to F(B)$ is an algebra map in $\D$. Moreover, if $\gamma: F\Rightarrow G$ is a monoidal natural transformation between lax monoidal functors $F,G:\C\to\D$, then $\gamma_A:F(A)\to G(A)$ is an algebra map in $\D$.

Now suppose that $A$ is an algebra in $\C$, $F:\C\to\D$ is a strong monoidal functor and $B$ is an algebra in $\D$. Moreover, suppose that $F$ admits a right adjoint $F^\ra$. Then, the adjunction $F\dashv F^{\ra}$ is a monoidal adjunction. In particular, the unit $\eta$ of the adjunction is a monoidal natural transformation. Thus, the isomorphism $\ozeta_F$ from \eqref{eq:zeta} restricts to an isomorphism between the set of algebra maps from $F(A)$ to $B$ and the set of algebra maps from $A$ to $F^\ra(B)$. This applies, for instance, when $F$ is a tensor functor between finite tensor categories. In this case, $F^\ra$ exists and $F\dashv F^{\ra}$ is a monoidal adjunction.

Suppose that $\C,\D$ and $\E$ are monoidal categories, and $F:\C\to\D$, $G:\D\to\E$ and $H:\C\to\E$ are strong monoidal functors that each admit right adjoints. Moreover, assume that $H \cong G\circ F$. Define the algebras
\[ A = H^\ra(\unit_\E) \cong F^\ra (G^\ra(\unit_\E)) \qquad \text{and} \qquad B = F^{\ra}(\unit_\D).\]
Let $\iota_G:G(\unit_\D)\xrightarrow{\sim}\unit_\E$ be the monoidal unit
isomorphism. Its adjunct is an algebra map
\[
  u:=\ozeta_G(\iota_G):\unit_\D\longrightarrow G^{\ra}(\unit_\E).
\]
Applying $F^{\ra}$ to $u$ gives another algebra map
$j:=F^{\ra}(u):B\to A$. The following result is \cite[Lemma~2.1]{shimizu2026transparent}. It states that factoring an algebra map through $j$ is equivalent to factoring its adjoint map through $G$.

\begin{lemma}\label{lem:alg-map-factorization}
Suppose that we are in the above setting. Take $C\in \C$ and $i\in \Hom_\C(C,A)$.
\begin{enumerate}
  \item For every $f:C\to B=F^\ra(\unit_\D)$, one has $\zeta_H(j\circ f)=G\bigl(\zeta_F(f)\bigr)$. 
  \item A map $i:C\to A$ factors through $j$ if and only if there exists a map $\varphi:F(C)\to\unit_\D$ such that $\zeta_H(i)=G(\varphi)$.
\end{enumerate}
The same statement holds for algebra maps.
\end{lemma}


\subsection{Restriction and extension of scalars}
Let $\B$ be a braided monoidal category with coequalizers such that the tensor product preserves coequalizers. Let $A$ be a commutative algebra in $\B$ and $B$ a subalgebra of $A$. Then, $B$ is also a commutative algebra in $\B$.

Define the functor $\Ind_B^A:\B_B\to\B_A$ by
$\Ind_B^A(M)=(M\otimes_B A,a_{M\otimes_B A}^r)$. Here the right action of
$A$ on $M\otimes_B A$ is the unique map satisfying
\begin{equation*}
  a_{M\otimes_B A}^r \circ (\pi \otimes \id_A) = \pi \circ (\id_M\otimes m_A) : M \otimes A \otimes A \to M\otimes_B A,
\end{equation*}
where $\pi: M\otimes A \to M\otimes_B A$ denotes the epimorphism onto the coequalizer. $\Ind_B^A$ is a left adjoint of the restriction functor $\Res_B^A:\B_A\to\B_B$. 
We note that restriction of scalars preserves local modules. Moreover, we have:
\begin{lemma}\label{lem:res-of-local}
Suppose that we are in the above setting. If $M\in \B_A^\loc$, then $\Res_B^A(M)\in\B_B^\loc$.
\end{lemma}
\begin{proof}
Take $(M,a_M^r)\in\B_A^\loc$ and write $\iota:B\hookrightarrow A$ for the
inclusion. Then
$\Res_B^A(M,a_M^r)=(M,a_M^r\circ(\id_M\otimes\iota))\in\B_B$.
Observe that
\begin{align*}
a_M^r\circ(\id_M\otimes \iota) \circ c_{B,M} \circ c_{M,B} & = a_M^r \circ c_{A,M} \circ c_{M,A} \circ (\id_M\otimes \iota) = a_M^r\circ (\id_M\otimes\iota) ,
\end{align*}
where the first equality uses naturality of braiding and the second that $(M,a_M^r)$ is a local $A$-module. Hence, we are done.
\end{proof}

Now suppose that $\B$ is a braided finite tensor category and $A$ is a commutative simple algebra in $\B$. Then, $B$ is also a commutative simple algebra in $\B$, see \cite[Lemma~2.3]{shimizu2026transparent}.
\begin{lemma}\label{lem:ind-surjective}
\begin{enumerate}
  \item $\Ind_B^A$ is a surjective tensor functor.
  \item $\Res_B^A$ is a faithful and exact lax monoidal functor.
\end{enumerate}  
\end{lemma}
\begin{proof}
Write $\D=\B_B$. The regular $A$-module is local, so
Lemma~\ref{lem:res-of-local} shows that $A|_B$ is local in $\D$. Its
canonical central lift makes $A$ a commutative algebra in $\Z(\D)$. By
\cite[Lemma~4.1]{schauenburg2001monoidal} and the standard transitivity of
module structures, we have a tensor equivalence
\[
  \D_A^\sigma\simeq\B_A
\]
under which the free $A$-module functor identifies with $\Ind_B^A$.

(1) Since $B$ is commutative simple, $\D$ is a finite tensor category. Moreover, $A$ is
simple as an algebra in $\D$, since every ideal in $\D$ gives, after
forgetting the $B$-module structure, an ideal of $A$ in $\B$. Therefore
Theorem~\ref{thm:central-simple-module-category} applies and shows that
$\Ind_B^A$ is surjective.

(2) Being a right adjoint to a strong monoidal functor, $\Res_B^A$ is lax
monoidal. Since $\Ind_B^A$ is surjective by part~(1),
Lemma~\ref{lem:tensor-functor-facts}(1)--(2) show that $\Res_B^A$ is faithful
and exact.
\end{proof}

The next lemma rewrites the locality of a restricted right-adjoint module as an equation involving the half-braiding of the central functor.

\begin{lemma}\label{lem:restriction-locality}
Let $F:\B\to\C$ be a central monoidal functor admitting a right
adjoint $R$. Let $A=R(\unit_\C)$, which is commutative by centrality. Let
$\iota:B\hookrightarrow A$ be a subalgebra, and set $q:=\varepsilon_{\unit_\C}\circ F(\iota):F(B)\longrightarrow\unit_\C$. 
Then, for $V\in\C$, the right $B$-module $\Res_B^A(R(V))$ is local if
and only if
\begin{equation}\label{eq:restriction-locality}
  (\id_V\otimes q)\circ\sigma_{B,V}\circ
  (\id_{F(B)}\otimes\varepsilon_V)
  =
  (q\otimes\id_V)\circ
  (\id_{F(B)}\otimes\varepsilon_V).
\end{equation}
If $\id_{F(B)}\otimes\varepsilon_V$ is an epimorphism, then this is
equivalent to
\begin{equation}\label{eq:restriction-locality-epic}
  (\id_V\otimes q)\circ\sigma_{B,V}
  =q\otimes\id_V.
\end{equation}
In particular, the latter conclusion holds if $\varepsilon_V$ is epic and
$F(B)\otimes(-)$ preserves epimorphisms.
\end{lemma}
\begin{proof}
The restricted action is
\[
  a_B^r:
  R(V)\otimes B
  \xrightarrow{\id_{R(V)}\otimes\iota}
  R(V)\otimes R(\unit_\C)
  \xrightarrow{R_2(V,\unit_\C)}
  R(V).
\]
Let $F_2(X,Y):F(X)\otimes F(Y)\to F(X\otimes Y)$ denote the
monoidal constraint of $F$. Then, the adjunction gives a bijection
\[
  \begin{aligned}
  \xi:\Hom_\B(R(V)\otimes B,R(V))
  &\longrightarrow
  \Hom_\C(FR(V)\otimes F(B),V),\\
  f&\longmapsto
  \varepsilon_V\circ F(f)\circ F_2(R(V),B).
  \end{aligned}
\]
Since the counit of a monoidal adjunction is monoidal, we have
\begin{equation}\label{eq:restriction-locality-mates}
  \xi(a_B^r)=\varepsilon_V\otimes q.
\end{equation}

The braided lift of $F$ gives $F(c_{X,Y})\circ F_2(X,Y) = F_2(Y,X)\circ\sigma_{X,F(Y)}$. Applying this twice gives the commutative square
\[
\begin{tikzcd}[column sep=huge,row sep=large]
FR(V)\otimes F(B)
  \arrow[rr,"\sigma_{B,FR(V)}\circ\sigma_{R(V),F(B)}"]
  \arrow[d,"{F_2(R(V),B)}"']
& & FR(V)\otimes F(B)
  \arrow[d,"{F_2(R(V),B)}"]
\\
F(R(V)\otimes B)
  \arrow[rr,"{F(c_{B,R(V)}\circ c_{R(V),B})}"']
& & F(R(V)\otimes B).
\end{tikzcd}
\]
Postcomposing the lower-right corner with
$\varepsilon_V\circ F(a_B^r)$ and using
\eqref{eq:restriction-locality-mates} yields
\begin{equation*}
  \xi\bigl(a_B^r\circ c_{B,R(V)}\circ c_{R(V),B}\bigr)
  =
  (\varepsilon_V\otimes q)\circ
  \sigma_{B,FR(V)}\circ\sigma_{R(V),F(B)}.
\end{equation*}

Since $\xi$ is injective, locality of $\Res_B^A(R(V))$ is equivalent to
\begin{equation}\label{eq:restriction-locality-before-naturality}
  \varepsilon_V\otimes q
  =
  (\varepsilon_V\otimes q)\circ
  \sigma_{B,FR(V)}\circ\sigma_{R(V),F(B)}.
\end{equation}
Precomposing with $\sigma_{R(V),F(B)}^{-1}$ and using naturality of the
half-braiding with respect to
$\varepsilon_V:FR(V)\to V$ and $q:F(B)\to\unit_\C$, we obtain
\begin{align*}
  (\varepsilon_V\otimes q)\circ
  \sigma_{R(V),F(B)}^{-1}
  &=(q\otimes\id_V)\circ
    (\id_{F(B)}\otimes\varepsilon_V),\\
  (\varepsilon_V\otimes q)\circ\sigma_{B,FR(V)}
  &=(\id_V\otimes q)\circ\sigma_{B,V}\circ
    (\id_{F(B)}\otimes\varepsilon_V).
\end{align*}
Thus \eqref{eq:restriction-locality-before-naturality} is equivalent to
\eqref{eq:restriction-locality}. If
$\id_{F(B)}\otimes\varepsilon_V$ is epic, it can be cancelled from
\eqref{eq:restriction-locality}, giving
\eqref{eq:restriction-locality-epic}.
\end{proof}

\subsection{Transparent subalgebras and local modules}

Let $i:\C\hookrightarrow\D$ be the inclusion of a tensor subcategory of a
finite tensor category. For $X\in\D$, we write $X\cap\C:= i i^{\ra}(X)$. In particular, if $A$
is an algebra in a braided finite tensor category $\B$, then
\[
  A':=A\cap\B'
\]
is the maximal transparent subalgebra of $A$. We record two results about $A'$
and local modules that will be used in the rest of the paper. These are
\cite[Theorem~1.1(1)--(2)]{shimizu2026transparent}.

\begin{theorem}\label{thm:transparent-local-inputs}
Let $\B$ be a braided finite tensor category and let $A$ be a commutative
simple algebra in $\B$. Set
\[
  K:=(F_A^-)^{\ra}(\unit_{\Z(\B_A)}),
  \qquad
  \D:=\im(F_A^-).
\]
Then the following statements hold.
\begin{enumerate}
  \item There is an isomorphism $K\cong A'$ of subalgebras of $A$, and the
  corestriction of $F_A^-$ induces a braided equivalence
  \[
    \D\simeq\overline{\B_{A'}}.
  \]
  \item The Frobenius--Perron dimension of the local-module category is
  \[
    \FP(\B_A^{\loc})
    =
    \frac{\FP(\B)\FP_{\B}(A')}{\FP_{\B}(A)^2}.
  \]
\end{enumerate}
\end{theorem}


\section{Left-modular elements of the lattice of tensor subcategories}\label{sec:modular}
Let $\C$ be a finite tensor category. Given tensor subcategories $\D_1,\D_2\subseteq \C$, we can construct two new tensor subcategories:
\begin{itemize}
  \item $\D_1\cap \D_2$: the intersection is defined to be the full subcategory of $\C$ consisting of objects that belong to both $\D_1$ and $\D_2$.
  \item $\langle \D_1,\D_2 \rangle_\otimes$: this is the smallest tensor subcategory of $\C$ containing the image of the functor $\D_1\boxtimes \D_2 \to\C$.
\end{itemize} 
Together these operations turn the set of tensor subcategories of $\C$ into a lattice with $\D_1\vee \D_2:= \langle \D_1,\D_2 \rangle_\otimes$ and $\D_1\wedge\D_2:= \D_1\cap \D_2$. We denote this lattice by $\Lten{\C}$.

A lattice is called \emph{modular} if $x\leq z$ implies $x\vee (y\wedge z) = (x\vee y)\wedge z$ for all $x,y,z$ in the lattice. The lattice $\Lten{\C}$ need not be modular in general: for example, the lattice $\Lten{\Vect_{D_8}}$ identifies with the subgroup lattice of the dihedral group $D_8$ and this lattice is not modular. 
We therefore consider the following notion.

\begin{definition}\cite{blass1997mobius}
  An element $a$ of a lattice $L$ is called \emph{left-modular} if for all $x,y \in L$ with $x\leq y$, we have $x\vee (a\wedge y) = (x\vee a)\wedge y$.
\end{definition}

In this section, we prove that central tensor subcategories are left-modular elements of the lattice of tensor subcategories of a finite tensor category $\C$. 

Let $\C$ be a finite tensor category and $\B$ a braided finite tensor category. We call $\B$ a central tensor subcategory of $\C$ if we have an injective tensor functor $F:\B\to \C$ which is central. Then $F$ admits a lift $\tF:\B\to \Z(\C)$ such that $F = \U\circ \tF$. We identify $\B$ with the tensor subcategory $F(\B)\subseteq\C$.

\begin{theorem}\label{thm:central-subcategory-left-modular}
For any tensor subcategories $\D_1,\D_2\subseteq \C$ with $\D_1\subseteq \D_2$ and $\B$ a central tensor subcategory of $\C$, we have
\begin{equation*}
  \D_1\vee (\B\wedge \D_2) = (\D_1\vee \B)\wedge \D_2.
\end{equation*}
\end{theorem}
In the proof below, we will use the fact that for any $\mathcal{X}\subseteq \C$, the following functor:
\[ \B \boxtimes \mathcal{X} \to \C, \qquad A \boxtimes X \to A\otimes X\]
is a tensor functor. This is because $\B$ is a central tensor subcategory of $\C$. Consequently, by \cite[Lemma~4.8]{shimizu2019non},
\begin{equation}\label{eq:fp-central-subcat}
  \FP(\B \wedge\X) \FP(\B \vee \X) = \FP(\B)\FP(\X) \quad \text{for any } \X\subseteq \C.
\end{equation}
\begin{proof}
Set
\[ \E = (\D_1 \vee \B)\wedge \D_2 \quad \text{and} \quad \F = \D_1 \vee (\B \wedge \D_2) .\]
As $\D_1\subseteq \D_1\vee\B$ and $\D_1\subseteq \D_2$, we have that $\D_1\subseteq \E$. Similarly, $\B\wedge \D_2 \subseteq (\D_1\vee \B)\wedge \D_2 = \E$. Thus, $\F\subseteq \E$. Therefore, it suffices to show that $\FP(\F)=\FP(\E)$. 

\underline{Calculation of $\FP(\E)$}: We make two observations. 
\begin{itemize}
  \item $\B \wedge \E = \B \wedge\D_2$: As $\E\subseteq \D_2$, the containment $\subseteq$ follows. On the other hand, we already noted that $\B\wedge \D_2 \subseteq \E$ and moreover $\B\wedge \D_2 \subseteq \B$, so $\B\wedge \D_2\subseteq \B\wedge\E$. 
  \item $\B\vee \E = \B \vee \D_1$: As $\E\supseteq \D_1$, the containment $\supseteq$ follows. On the other hand, both $\B$ and $\E$ are subcategories of $\B\vee \D_1$, so the containment $\subseteq$ follows.
\end{itemize} 
Now, we can use \eqref{eq:fp-central-subcat} with $\X=\E$ along with the above two observations to get
\begin{equation*}
   \FP(\B\wedge \D_2) \FP(\B \vee \D_1) = \FP(\E) \FP(\B).
\end{equation*}
Using \eqref{eq:fp-central-subcat}, now with $\X=\D_1$, we get that
\begin{equation*}
  \FP(\B\wedge\D_1)\FP(\B \vee \D_1) = \FP(\B) \FP(\D_1).
\end{equation*}
  Combining the above observations yields:
\begin{equation}\label{eq:FP-E}
  \FP(\E) = \frac{\FP(\B \wedge \D_2) \FP(\D_1)}{\FP(\B \wedge \D_1)}.
\end{equation}

\underline{Calculation of $\FP(\F)$}:  As $\B\subseteq \C$ is central, so is $\B\wedge \D_2\subseteq \D_2$. Thus, we can use \eqref{eq:fp-central-subcat} with $\X=\D_1$ and $\B$ replaced by $\B\wedge \D_2$ to get
\begin{equation}\label{eq:FP-F}
  \FP(\D_1) \FP(\B\wedge \D_2)  = \FP(\D_1\vee (\B\wedge \D_2)) \FP(\D_1\wedge(\B \wedge \D_2)) = \FP(\F) \FP(\B\wedge\D_1).
\end{equation}
From \eqref{eq:FP-E} and \eqref{eq:FP-F}, it follows that $\FP(\E) = \FP(\F)$. As $\F\subseteq \E$, we conclude that $\F=\E$.
\end{proof}

\begin{corollary}\label{cor:braided-subcategory-lattice-modular}
If $\C$ is a braided finite tensor category, then $\Lten{\C}$ is modular.
\end{corollary}
\begin{proof}
For every tensor subcategory $\B\subseteq\C$, the inclusion is central using
the canonical lift
\[
  \B\longrightarrow \C \longrightarrow \Z(\C),
  \qquad
  X\longmapsto (X,c_{X,-}).
\]
By Theorem~\ref{thm:central-subcategory-left-modular}, every
element of $\Lten{\C}$ is left-modular. Hence $\Lten{\C}$ is modular.
\end{proof}
In the fusion case, this result is \cite[Lemma~5.6]{drinfeld2010braided}.

\begin{remark}
For tensor subcategories $\D,\E\subseteq\C$, the braiding makes the
multiplication functor $\D\boxtimes\E\to\C$ a tensor functor. Using
\eqref{eq:fp-central-subcat}, it follows that the function $v(\D):=\log\FP(\D)$ satisfies
\[
  v(\D\vee\E)+v(\D\cap\E)
  =
  v(\D)+v(\E).
\]
Thus, $v$ is a strictly increasing valuation on $\Lten{\C}$. It is not in
general the rank function of the lattice, since its increments along covers
need not be equal to one.  
\end{remark}

\begin{remark}
There is also a natural order-reversing operation on $\Lten{\C}$. Let $\C'$ be
the M\"uger center of $\C$. For every tensor subcategory $\D\subseteq\C$,
\cite[Theorem~4.9]{shimizu2019non} gives
\[
  \D''=\D\vee\C'.
\]
It follows that taking the M\"uger centralizer in $\C$ is an order-reversing involution on the interval $[\C',\C]_{\otimes}$. 
If $\C$ is nondegenerate, then $\C'=\Vect$, and hence taking the M\"uger centralizer in $\C$ is an order-reversing involution on the whole lattice $\Lten{\C}$.  
\end{remark}


\section{The Galois correspondence}\label{sec:galois-connection}
Let $\B$ be a braided finite tensor category and let $A$ be a commutative
simple algebra in $\B$. Consider the following lattices:
\begin{align*}
    \Lalg{A}
  = \{\text{algebras }B\mid B\xhookrightarrow{i_B}A\}, \qquad \Lten{\B_A}
  = \{\text{tensor subcategories }\E\mid
      \E\xhookrightarrow{\iota_\E}\B_A\}.
\end{align*}
The two lattices are ordered by inclusion. 
Thus $\E_1\leq\E_2$ if the inclusion of $\E_1$ into $\B_A$ factors
through $\E_2$, and $B_1\leq B_2$ if the inclusion of $B_1$ into $A$ factors
through $B_2$. In this section, we construct order-reversing maps
$\alpha:\Lten{\B_A}\to\Lalg{A}$ and
$\beta:\Lalg{A}\to\Lten{\B_A}$, in other words, they form a Galois connection.  Here, $\alpha$ is constructed using
relative Drinfeld centers and $\beta$ using induction from
local modules. 

Set $A'=A\cap \B'$. Consider the following intervals:
\begin{equation*}
[A',A]_{\mathrm{alg}}
  = \{B\in\Lalg{A}\mid A'\hookrightarrow B\}, \qquad 
[\B_A^\loc,\B_A]_{\otimes}
  = \{\E\in\Lten{\B_A}\mid \B_A^\loc\subseteq\E\}.
\end{equation*} 
We prove formulae for the Frobenius-Perron dimension of $\alpha(\E)$ 
and $\beta(B)$. Using these, we 
show that the Galois connection $(\alpha,\beta)$ restricts to a Galois correspondence between the intervals $[A',A]_{\mathrm{alg}}$ and $[\B_A^\loc,\B_A]_{\otimes}$.


\subsection{The construction \texorpdfstring{$\alpha$}{alpha}} \label{subsec:alpha}
We begin with the map from tensor subcategories to subalgebras. 
Let $\E \in \Lten{\B_A}$. Define $\Q_\E$ as the composition
\begin{equation*}
  \Q_\E: \B \xrightarrow{\tF_A} \Z(\B_A)=\Z(\B_A;\B_A) \xrightarrow{\R_{\E}} \Z(\E;\B_A).
\end{equation*}
When clear from context, we write $\Q=\Q_\E$. We define
\begin{equation*}
  \alpha(\E) = \Q^{\ra}(\unit_{\Z(\E;\B_A)}).
\end{equation*} 
As $\Q_\E^\ra$ is lax-monoidal, it follows that $\alpha(\E)$ is an algebra in $\B$. The following lemma proves that $\alpha$ is well-defined.

\begin{lemma}\label{lem:alpha-basic-properties}
\begin{enumerate}
  \item $\alpha(\E)$ is canonically a subalgebra of $A$.
  \item $\alpha$ is order reversing.
  \item $\alpha(\B_A) = A' = A\cap\B'$ as subalgebras of $A$.
  Consequently, $A'\subseteq\alpha(\E)$ for every
  $\E\in\Lten{\B_A}$.
  \item $\alpha(\B_A^\loc) = A$, and hence
  $\alpha(\Vect) = A$.
\end{enumerate}
\end{lemma}

\begin{proof}
(1) Let $\U_\E$ and $\I_\E$ be the forgetful functor and its right adjoint
from \S\ref{ssec:relative-drinfeld-centers}, with
$\C=\B_A$ and $\D=\E$. By construction,
\[
  \U_\E\circ\Q_\E=F_A.
\]
By \S\ref{ssec:alg-mon-adj}, the adjunct of $\id_A$ under
$\U_\E\dashv\I_\E$ is a morphism of algebras
\[
  \unit_{\Z(\E;\B_A)}\longrightarrow \I_\E(A).
\]
It is non-zero and hence monic, since the tensor unit is simple. Applying the
left exact lax monoidal functor
$\Q_\E^\ra$ therefore gives a monomorphism of algebras
\[
  \alpha(\E)=\Q_\E^\ra(\unit_{\Z(\E;\B_A)})
  \hookrightarrow
  \Q_\E^\ra\I_\E(A)
  \cong F_A^\ra(A)
  \cong A.
\]
We denote this canonical inclusion by
$i_{\alpha(\E)}:\alpha(\E)\hookrightarrow A$.

(2) Suppose that $\E_1\leq \E_2$ via $F:\E_1\hookrightarrow\E_2$. Let
\[
  \R_F:\Z(\E_2;\B_A)\to\Z(\E_1;\B_A)
\]
be restriction of half-braidings, as in \eqref{eq:ZC-to-ZDC}. Then
$\Q_{\E_1}=\R_F\circ\Q_{\E_2}$. As in part~(1), the unit of
$\R_F\dashv \R_F^\ra$ at the tensor unit is a monomorphism of algebras
\[
  \unit_{\Z(\E_2;\B_A)}
  \hookrightarrow
  \R_F^\ra\bigl(\unit_{\Z(\E_1;\B_A)}\bigr).
\]
Applying $\Q_{\E_2}^\ra$ gives a monic algebra map
\[
  \rho_{\E_2,\E_1}:\alpha(\E_2)
  \hookrightarrow
  \Q_{\E_2}^\ra \R_F^\ra
  \bigl(\unit_{\Z(\E_1;\B_A)}\bigr)
  \cong\alpha(\E_1).
\]
Since $\U_{\E_1}\circ \R_F=\U_{\E_2}$, functoriality of the units for
composite adjunctions gives $i_{\alpha(\E_1)}\circ\rho_{\E_2,\E_1} = i_{\alpha(\E_2)}$. 
This equality is precisely the compatibility required by the order on
$\Lalg{A}$.
Hence $\alpha(\E_2)\leq\alpha(\E_1)$, so $\alpha$ is order reversing.

(3) For $\E=\B_A$, the functor $\R_\E$ is the identity and
$\Q_\E=F_A^-$. Hence Theorem~\ref{thm:transparent-local-inputs}(1) gives
\[
  \alpha(\B_A)
  =(F_A^-)^\ra(\unit_{\Z(\B_A)})
  \cong A\cap\B'. 
\]
For every $\E\in\Lten{\B_A}$, the inclusion $\E\subseteq\B_A$ and part~(2)
give $A'=\alpha(\B_A)\subseteq\alpha(\E)$.

(4) The algebra map $\id_A:A\to A$ corresponds under the free-module adjunction
to
\[
  q_A=m_A:F_A(A)=A\otimes A\longrightarrow A.
\]
Apply \cite[Lemma~2.2]{shimizu2026transparent} with $L=A$. Every
$M\in\B_A^\loc$ is local as an $A$-module, so the lemma shows that $q_A$
preserves the inverse-braiding half-braiding with every object of
$\B_A^\loc$. Thus $q_A$ is a morphism in $\Z(\B_A^\loc;\B_A)$ from
$\Q_{\B_A^\loc}(A)$ to the tensor unit.
Lemma~\ref{lem:alg-map-factorization} therefore shows that $\id_A$ factors
through the canonical inclusion
\[
  i_{\alpha(\B_A^\loc)}:\alpha(\B_A^\loc)\hookrightarrow A.
\]
This inclusion is thus an isomorphism.
Since $\Vect\subseteq\B_A^\loc$, part~(2) gives
$A=\alpha(\B_A^\loc)\subseteq\alpha(\Vect)\subseteq A$.
\end{proof}

We now restrict to $\E\in[\B_A^\loc,\B_A]_{\otimes}$ and give a formula for the Frobenius--Perron dimension of $\alpha(\E)$. First we need a technical lemma. Set
\[
  \D:=\im(F_A^-)=\im(\tF_A)\subseteq\Z(\B_A).
\]
We factor $\tF_A$ through its image as $\B \xrightarrow{h} \D\xrightarrow{i} \Z(\B_A)$,  where $h$ is surjective and $i$ is the inclusion. For
$\E\in[\B_A^\loc,\B_A]_{\otimes}$, set
$J_{\E} := \R_{\E}^{\ra}(\unit_{\Z(\E;\B_A)})$. 
The next lemma proves $J_{\E}\in\D$. This is precisely the hypothesis needed
to apply Lemma~\ref{lem:FPdim} to the composite
\[
  \D\xrightarrow{i}\Z(\B_A)
  \xrightarrow{\R_{\E}}\Z(\E;\B_A).
\]

\begin{lemma}\label{lem:restriction-adjoint-in-minus-image}
For every $\E\in[\B_A^\loc,\B_A]_{\otimes}$, the object $J_{\E}$ belongs
to $\D$.
\end{lemma}
\begin{proof}
Set
\[
  \E_0:=\B_A^\loc,
  \qquad
  J_0:=J_{\E_0},
  \qquad
  X:=i^{\ra}(J_0).
\]
The counit of $i\dashv i^{\ra}$ gives a monomorphism $i(X)\hookrightarrow J_0$. 
Since $\tF_A=i\circ h$, we have
\[
  h^{\ra}(X)
  \cong \tF_A^{\ra}(J_0)
  =\alpha(\E_0)
  \cong A,
\]
where the last isomorphism follows from
Lemma~\ref{lem:alpha-basic-properties}(4). Since $i$ is fully faithful, the unit
of $i\dashv i^{\ra}$ is an isomorphism. Hence
$i^{\ra}(\unit_{\Z(\B_A)})\cong\unit_{\D}$, and
Theorem~\ref{thm:transparent-local-inputs}(1) gives
\begin{equation}\label{eq:alpha-A'}
  h^{\ra}(\unit_{\D})
  \cong (F_A^-)^{\ra}(\unit_{\Z(\B_A)})
  \cong A'.  
\end{equation}
It follows from \cite[Lemma~6.2.4]{etingof2016tensor} that $\FP(\D)=\frac{\FP(\B)}{\FP_{\B}(A')}$ and $\FP_{\D}(X) = \frac{\FP_{\B}(A)}{\FP_{\B}(A')}$. 

We now compute the Frobenius--Perron dimension of $J_0$. Since
$\R_{\E_0}$ is surjective, \cite[Lemma~6.2.4]{etingof2016tensor} and
\cite[Lemma~4.6]{shimizu2019non} give
\begin{equation*}
  \FP_{\Z(\B_A)}(J_0) = \frac{\FP(\Z(\B_A))}{\FP(\Z(\E_0;\B_A))} = \frac{\FP(\B_A)}{\FP(\E_0)} = \frac{\FP_{\B}(A)}{\FP_{\B}(A')}.
\end{equation*}
The last equality follows from \cite[Lemma~5.15]{shimizu2026commutative}
and Theorem~\ref{thm:transparent-local-inputs}(2). The tensor functor $i$
preserves Frobenius--Perron dimensions. Thus the monomorphism
$i(X)\hookrightarrow J_0$ has source and target of the same dimension. It
is therefore an isomorphism, and $J_0\in\D$.

For a general $\E\in[\B_A^\loc,\B_A]_{\otimes}$, the inclusion
$\E_0\subseteq\E$ and the
restriction of half-braidings give a monomorphism $J_{\E}\hookrightarrow J_0$ 
by the same argument as in Lemma~\ref{lem:alpha-basic-properties}(2).
Since $\D$ is closed under subobjects, $J_{\E}\in\D$.
\end{proof}

\begin{proposition}\label{prop:FP-alpha-D}
For every $\E\in[\B_A^\loc,\B_A]_{\otimes}$, we have
\[
  \FP_{\B}(\alpha(\E))
  =
  \frac{
    \FP(\B)\FP_{\B}(A')
  }{
    \FP_{\B}(A)\FP(\E)
  }.
\]
\end{proposition}
\begin{proof}
By Lemma~\ref{lem:restriction-adjoint-in-minus-image}, $J_{\E}$ belongs
to $\D$. We apply Lemma~\ref{lem:FPdim} to the tensor functors
\[
  \D\xrightarrow{i}\Z(\B_A)
  \xrightarrow{\R_{\E}}\Z(\E;\B_A).
\]
The functor $i$ is injective, the functor $\R_{\E}$ is surjective, and
$\R_{\E}^{\ra}(\unit)=J_{\E}$ belongs to $\im(i)$. Hence
\begin{equation}\label{eq:alpha-FPdim-lemma}
  \FP_{\D}(J_{\E})
  =\frac{\FP(\Z(\B_A))}{\FP(\Z(\E;\B_A))}
  =\frac{\FP(\B_A)}{\FP(\E)}
  =\frac{\FP(\B)}{\FP_{\B}(A)\FP(\E)}.
\end{equation}
Here the second equality follows from
\cite[Lemma~4.6]{shimizu2019non}, and the last one follows from
\cite[Lemma~5.15]{shimizu2026commutative}.

Since $\alpha(\E)=\tF_A^{\ra}(J_{\E})\cong h^{\ra}(J_{\E})$, another
application of \cite[Lemma~6.2.4]{etingof2016tensor} gives
\[
  \FP_{\B}(\alpha(\E))
  =\FP_{\B}(h^{\ra}(\unit))\FP_{\D}(J_{\E})
  \stackrel{\eqref{eq:alpha-A'}}{=} \FP_{\B}(A')\FP_{\D}(J_{\E}).
\]
The result now follows from \eqref{eq:alpha-FPdim-lemma}.
\end{proof}


\subsection{The construction \texorpdfstring{$\beta$}{beta}}
We now construct the map in the opposite direction. Starting from a
subalgebra $B\subseteq A$, we induce local $B$-modules to $A$-modules and
take the image of this induction functor.

Let $B\in\Lalg{A}$. Recall the induction functor $\Ind_B^A = -\otimes_B A: \B_B \to \B_A$. Let $L$ denote the restriction of $\Ind_B^A$ to $\B_B^\loc$. We define
\begin{equation*}
  \beta(B) = \im(L) \subseteq \B_A.
\end{equation*}
Also, define
\begin{equation*}
  \T_B = \{ M\in\B_A \ \mid \ \Res_B^A(M) \in \B_B^\loc\}.
\end{equation*}
Note that $A\in\T_B$. The category $\beta(B)$ is defined through induction, while $\T_B$ describes the same subcategory intrinsically through locality after
restriction to $B$. The next lemma shows that these two descriptions agree.

\begin{lemma}\label{lem:T-beta-same}
  As full subcategories of $\B_A$, $\T_B = \beta(B)$.
\end{lemma}
\begin{proof}
Take $M\in\T_B$. By Lemma~\ref{lem:ind-surjective}(2), the counit
\[
  \Ind_B^A(\Res_B^A(M))\longrightarrow M
\]
is an epimorphism. Since $\Res_B^A(M)$ is local, this shows that
$M\in\beta(B)$.

Conversely, let $M\in\beta(B)$. Then $M$ is a subquotient of
$\Ind_B^A(N)$ for some $N\in\B_B^\loc$. The functor $\Res_B^A$ is exact,
and
\[
  \Res_B^A(\Ind_B^A(N))=N\otimes_B A.
\]
Since $A$ is the tensor unit of $\B_A^\loc$,
Lemma~\ref{lem:res-of-local} gives
$\Res_B^A(A)\in\B_B^\loc$. The category $\B_B^\loc$ is monoidal, and hence
\[
  N\otimes_B A
  =N\otimes_B\Res_B^A(A)
  \in\B_B^\loc.
\]
Since $\B_B^\loc$ is closed under subquotients,
$\Res_B^A(M)\in\B_B^\loc$. Hence $M\in\T_B$.
\end{proof}

\begin{lemma}\label{lem:B-A-loc-in-TB}
\begin{enumerate}
  \item $\B_A^\loc$ is contained in $\T_B$.
  \item $\beta$ is order reversing.
\end{enumerate} 

\end{lemma}
\begin{proof}
(1) is a direct consequence of Lemma~\ref{lem:res-of-local}.

(2) Suppose that $B_1\leq B_2$. If $M\in\T_{B_2}$, then
$\Res_{B_2}^A(M)$ is a local $B_2$-module. Restricting further along
$B_1\hookrightarrow B_2$ preserves locality by
Lemma~\ref{lem:res-of-local}. Since $\Res_{B_1}^A = \Res_{B_1}^{B_2}\circ\Res_{B_2}^A$, we obtain $M\in\T_{B_1}$. Hence
\[
  \beta(B_2)=\T_{B_2}\subseteq\T_{B_1}=\beta(B_1),
\]
so $\beta$ is order reversing.
\end{proof}

Together, Lemmas~\ref{lem:T-beta-same} and \ref{lem:B-A-loc-in-TB} show that
$\beta:\Lalg{A}\to\Lten{\B_A}$ is well defined and that its image is
contained in $[\B_A^\loc,\B_A]_{\otimes}$. We next compute its
Frobenius--Perron dimension.
\begin{proposition}\label{prop:FP-beta-B}
We have
\[
  \FP(\beta(B))
  =
  \frac{
    \FP(\B)\FP_{\B}(B\cap\B')
  }{
    \FP_{\B}(B)\FP_{\B}(A)
  }.
\]
\end{proposition}
\begin{proof}
The adjunction $\Ind_B^A\dashv\Res_B^A$ restricts to an adjunction
\[
  L:\B_B^\loc\rightleftarrows\beta(B):L^\ra,
\]
where $L^\ra$ is the restriction of $\Res_B^A$ and
$L^\ra(\unit)=A$. The functor $L$ is surjective and is identified with the
free module functor
\[
  \B_B^\loc\longrightarrow (\B_B^\loc)_A\simeq\beta(B).
\]
Note that $L^\ra$ is the restriction of
$\Res_B^A$ and the tensor unit of $\beta(B)$ is $A$. Hence
\cite[Lemma~6.2.4]{etingof2016tensor} gives
\[
  \FP(\beta(B))
  =
  \frac{\FP(\B_B^\loc)}
  {\FP_{\B_B^\loc}(\Res_B^A(A))}.
\]
By Theorem~\ref{thm:transparent-local-inputs}(2),
\[
  \FP(\B_B^\loc)
  =
  \frac{
    \FP(\B)\FP_{\B}(B\cap\B')
  }{
    \FP_{\B}(B)^2
  }.
\]
The inclusion $\B_B^\loc\hookrightarrow\B_B$ preserves
Frobenius--Perron dimensions. Also, applying
\cite[Lemma~6.2.4]{etingof2016tensor} to $F_B:\B\to\B_B$ gives
\[
  \FP_{\B_B^\loc}(\Res_B^A(A)) =
  \FP_{\B_B}(A)=\frac{\FP_{\B}(A)}{\FP_{\B}(B)}.
\]
Substituting these formulas proves the result.
\end{proof}


\subsection{The Galois connection and its consequences}
In this section, we prove the Galois connection. As before, $\B$ is a braided finite tensor category and $A$ is a simple commutative algebra in $\B$. We first recall the lattice operations in the two lattices $\Lalg{A}$ and $\Lten{\B_A}$. 

For $B_1,B_2\in\Lalg{A}$, write
$\langle B_1,B_2\rangle_{\mathrm{alg}}$ for the subalgebra of $A$ generated
by $B_1$ and $B_2$. For $\E_1,\E_2\in\Lten{\B_A}$, write
$\langle\E_1,\E_2\rangle_{\otimes}$ for the tensor subcategory generated by
$\E_1$ and $\E_2$. The meets in $\Lalg{A}$ and $\Lten{\B_A}$ are
intersections, and the joins are these generated objects. The intervals
$[A',A]_{\mathrm{alg}}$ and $[\B_A^\loc,\B_A]_{\otimes}$ are sublattices of
the corresponding full lattices. For subalgebras, the required intersection
is already a finite intersection because $A$ has finite length.

\begin{theorem}\label{thm:general-galois-connection}
The maps $\alpha:\Lten{\B_A}\to\Lalg{A}$ and
$\beta:\Lalg{A}\to\Lten{\B_A}$ form an order-reversing Galois connection.
Explicitly, for every $B\in\Lalg{A}$ and $\E\in\Lten{\B_A}$,
\[
  \E\subseteq\beta(B)
  \quad\Longleftrightarrow\quad
  B\subseteq\alpha(\E).
\]
Their closure operators are
\[
  \alpha\beta(B)=\langle B,A'\rangle_{\mathrm{alg}},
  \qquad
  \beta\alpha(\E)=\langle\E,\B_A^\loc\rangle_\otimes.
\]
Thus $\alpha\beta(B)=B$ if and only if $A'\subseteq B$, while
$\beta\alpha(\E)=\E$ if and only if $\B_A^\loc\subseteq\E$. Hence their
fixed points are $[A',A]_{\mathrm{alg}}$ and
$[\B_A^\loc,\B_A]_{\otimes}$, respectively.
\end{theorem}
\begin{proof}
We prove the theorem in three stages. First, we translate the inclusion
$B\subseteq\alpha(\E)$ into the locality of $\Res_B^A(M)$ for every
$M\in\E$; this gives the Galois equivalence. We then compute
$\alpha\beta(B)$ using Proposition~\ref{prop:FP-beta-B}. Finally, we compute
$\beta\alpha(\E)$ first for $\E\in [\B_A^\loc,\B_A]_\otimes$ and 
then for arbitrary $\E\in\Lten{\B_A}$ by adjoining $\B_A^\loc$.

Let
\[
  q_B:F_A(B)=B\otimes A\longrightarrow A,
  \qquad
  q_B=m_A\circ(i_B\otimes\id_A),
\]
be the right $A$-module map corresponding to the algebra map
$i_B:B\to A$ under the adjunction $F_A\dashv F_A^\ra$.

Apply Lemma~\ref{lem:alg-map-factorization} with
$(F,G,H)=(\Q_\E,\U_\E,F_A)$. The identity
$\U_\E\circ\Q_\E=F_A$ and the identification of the map $j$ with the
canonical inclusion $i_{\alpha(\E)}:\alpha(\E)\hookrightarrow A$ were
established in Lemma~\ref{lem:alpha-basic-properties}(1). Hence
the algebra-map assertion of Lemma~\ref{lem:alg-map-factorization} shows that
$B\subseteq\alpha(\E)$ if and only if $q_B$ lifts to an algebra morphism
\[
  \Q_\E(B)\longrightarrow\unit_{\Z(\E;\B_A)}
\]
in $\Z(\E;\B_A)$. By the definition of the relative center, this means that
$q_B$ preserves the inverse-braiding half-braiding with every $M\in\E$.
For each such $M$, apply \cite[Lemma~2.2]{shimizu2026transparent} with $L=B$.
The lemma identifies this condition with locality of $\Res_B^A(M)$. Therefore,
\[
  B\subseteq\alpha(\E)
  \quad\Longleftrightarrow\quad
  \E\subseteq\T_B.
\]
Finally, Lemma~\ref{lem:T-beta-same} gives $\T_B=\beta(B)$ and proves the
Galois equivalence. Thus the relative-center factorization defining $\alpha$ is 
equivalent to the locality condition defining $\beta$. Next, we identify the two closure operators.

We first compute $\alpha\beta$. Fix $B\in\Lalg{A}$, set
$C:=\langle B,A'\rangle_{\mathrm{alg}}$, and let $\E_0:=\beta(B)$. The
Galois connection gives $B\subseteq\alpha(\E_0)$. Since
$\E_0\subseteq\B_A$, Lemma~\ref{lem:alpha-basic-properties}(2) gives
$A'=\alpha(\B_A)\subseteq\alpha(\E_0)$. Thus
$C\subseteq\alpha(\E_0)$, and another application of the Galois connection
gives $\E_0\subseteq\beta(C)$. The reverse inclusion follows from
$B\subseteq C$. Hence $\beta(C)=\beta(B)$.

Set $C_1:=\alpha\beta(C)$. The Galois connection gives $C\subseteq C_1$.
Thus both $C$ and $C_1$ contain $A'$. Since $A'$ is the maximal transparent
subalgebra of $A$, we have $C\cap\B'=C_1\cap\B'=A'$. The triangular identity
gives $\beta(C_1)=\beta(C)$. Proposition~\ref{prop:FP-beta-B} now gives
$\FP_{\B}(C)=\FP_{\B}(C_1)$. Therefore $C\subseteq C_1$ is an isomorphism.
Since $\beta(C)=\beta(B)$, we conclude that
$\alpha\beta(B)=\alpha\beta(C)=C=\langle B,A'\rangle_{\mathrm{alg}}$.

We next compute $\beta\alpha$ first on the interval
$[\B_A^\loc,\B_A]_{\otimes}$. Let
$\E_0\in[\B_A^\loc,\B_A]_{\otimes}$. Lemma~\ref{lem:alpha-basic-properties}(3) gives
$A'\subseteq\alpha(\E_0)$, and hence
$\alpha(\E_0)\cap\B'=A'$. Propositions~\ref{prop:FP-beta-B} and
\ref{prop:FP-alpha-D} give
\[
  \FP(\beta\alpha(\E_0))
  =\frac{\FP(\B)\FP_{\B}(A')}
  {\FP_{\B}(\alpha(\E_0))\FP_{\B}(A)}
  =\FP(\E_0).
\]
The Galois connection gives $\E_0\subseteq\beta\alpha(\E_0)$. This inclusion
is an equivalence by \cite[Proposition~6.3.3]{etingof2016tensor}. Thus
$\beta\alpha(\E_0)=\E_0$.

For an arbitrary $\E$, the closure must contain both $\E$ and
$\B_A^\loc$. The natural candidate is therefore the saturation
$\E^{\mathrm{sat}}=\langle\E,\B_A^\loc\rangle_\otimes$. The Galois
connection gives $\E\subseteq\beta\alpha(\E)$, while
Lemma~\ref{lem:B-A-loc-in-TB}(1) gives
$\B_A^\loc\subseteq\beta\alpha(\E)$. Hence
$\E^{\mathrm{sat}}\subseteq\beta\alpha(\E)$. The order reversing property tells 
that $\alpha(\E^{\mathrm{sat}})\subseteq\alpha(\E)$, and the Galois equivalence
applied to the preceding inclusion gives the reverse containment. 
Therefore, we get that $\alpha(\E^{\mathrm{sat}})=\alpha(\E)$. Lastly, since
$\E^{\mathrm{sat}}\in[\B_A^\loc,\B_A]_{\otimes}$, the interval case gives
$\beta\alpha(\E)=\beta\alpha(\E^{\mathrm{sat}})=\E^{\mathrm{sat}}$.
\end{proof}

\begin{remark}
In particular, $\E\subseteq\beta\alpha(\E)$ and
$B\subseteq\alpha\beta(B)$. Applying the order-reversing maps once more gives
$\alpha\beta\alpha=\alpha$ and $\beta\alpha\beta=\beta$.
\end{remark}

\begin{corollary}\label{cor:general-lattice-anti-isomorphism}
The restrictions of $\alpha$ and $\beta$ are mutually inverse lattice
anti-isomorphisms
\[
  [\B_A^{\loc},\B_A]_{\otimes}
  \underset{\beta}{\overset{\alpha}{\rightleftarrows}}
  [A',A]_{\mathrm{alg}}.
\]
For $\E_1,\E_2\in\Lten{\B_A}$ and
$B_1,B_2\in\Lalg{A}$, we have
\begin{align}\label{eq:general-alpha-beta-join}
  \alpha(\langle\E_1,\E_2\rangle_{\otimes})
  =\alpha(\E_1)\cap\alpha(\E_2),
  \quad
  \beta(\langle B_1,B_2\rangle_{\mathrm{alg}})
  =\beta(B_1)\cap\beta(B_2).
  \end{align}
If $\E_1,\E_2\in[\B_A^\loc,\B_A]_{\otimes}$ and
$B_1,B_2\in[A',A]_{\mathrm{alg}}$, then also
\begin{align*}
  \alpha(\E_1\cap\E_2)
  =\langle\alpha(\E_1),\alpha(\E_2)\rangle_{\mathrm{alg}},
  \qquad
  \beta(B_1\cap B_2)
  =\langle\beta(B_1),\beta(B_2)\rangle_{\otimes}.
\end{align*}
\end{corollary}
\begin{proof}
Theorem~\ref{thm:general-galois-connection} identifies
$[A',A]_{\mathrm{alg}}$ and $[\B_A^\loc,\B_A]_{\otimes}$ as the two
fixed-point lattices. This proves the anti-isomorphism. The formulas
\eqref{eq:general-alpha-beta-join} follow directly
from the Galois connection. The other two formulas are the meet and join
formulas for the anti-isomorphism.
\end{proof}

\begin{corollary}\label{cor:full-lattice-anti-isomorphism}
\begin{enumerate}
  \item The maps $\alpha$ and $\beta$ are mutually inverse on
  $[\B_A^\loc,\B_A]_{\otimes}$ and $\Lalg{A}$ if and only if
  $A'\cong\unit$.
  \item They are mutually inverse on $\Lten{\B_A}$ and $\Lalg{A}$ if and
  only if $A'\cong\unit$ and
  $\B_A^\loc\simeq\Vect$.
\end{enumerate}
\end{corollary}

\begin{proof}
The first statement follows from
$\alpha\beta(B)=\langle B,A'\rangle_{\mathrm{alg}}$ by applying the formula
to the unit subalgebra. For the second statement, both closure operators must
be the identity. The first one is the identity precisely when
$A'\cong\unit$. The second one is the identity precisely when
$\B_A^\loc\simeq\Vect$, as seen by applying it to $\Vect$.
\end{proof}

\begin{example}[The symmetric case]\label{ex:symmetric-case}
Suppose that $\B$ is symmetric. Then $A'=A$ and every $A$-module is local.
Hence $\alpha(\E)=A$ for every $\E\in\Lten{\B_A}$ and
$\beta(B)=\B_A$ for every $B\in\Lalg{A}$. Both maps are constant, and
their fixed-point anti-isomorphism is $\{A\}\leftrightarrow\{\B_A\}$. In
particular, $\beta$ does not distinguish the proper subalgebras of $A$.
\end{example}

\begin{theorem}\label{thm:nondegenerate-lattice-isomorphism}
Let $\B$ be a nondegenerate braided finite tensor category and let $A$ be a
simple commutative algebra in $\B$. Then
\[
  \beta:\Lalg{A}^{\mathrm{op}}
  \xrightarrow{\ \sim\ }
  [\B_A^{\loc},\B_A]_{\otimes}
\]
is an isomorphism of lattices with inverse $\alpha$.
\end{theorem}

\begin{proof}
Since $\B$ is nondegenerate, $A'=A\cap\B'\cong\unit$. The result follows
from Corollary~\ref{cor:full-lattice-anti-isomorphism}(1).
\end{proof}


\subsection{Frobenius extensions}
An extension $B\subseteq A$ of algebras is called a \textit{Frobenius extension} if $\Ind_B^A$ is a Frobenius functor, that is, its left and right adjoint are isomorphic. As a special case, an algebra $A$ is called \emph{Frobenius} if the free module functor $F_A:\B\to\B_A$ is a Frobenius functor.

The construction of $\beta(B)$ already involves the restriction of the
induction functor $\Ind_B^A$ to local $B$-modules. It is therefore natural to
ask whether the Frobenius property of the extension $B\subseteq A$ can be
detected from the tensor category $\beta(B)$. The next theorem gives such a
criterion under the assumption that $\B_B^\loc$ is unimodular.

\begin{theorem}\label{thm:frob-extn-unimodular}
  Let $B\subseteq A$ be a subalgebra of a simple commutative algebra. If $\B_B^\loc$ is unimodular, then $B\subseteq A$ is a Frobenius extension if and only if $\beta(B)$ is unimodular.
\end{theorem}
Note that, by \cite[Proposition~4.5]{etingof2004analogue}, $\B_B^\loc$ is unimodular when it is factorizable. Moreover, as being factorizable is equivalent to being nondegenerate \cite{shimizu2019non}, we can use the description $(\B_B^\loc)'\simeq \B'_{B'}$ from \cite[Theorem~1.1(3)]{shimizu2026transparent}, to obtain sufficient conditions for unimodularity of $\B_B^\loc$. For instance, nondegeneracy of $\B$ suffices.  
\begin{proof}
As explained after Lemma~\ref{lem:T-beta-same}, the functor
\[
  L:\B_B^\loc\longrightarrow\beta(B)
\]
is the free functor associated to the algebra $\Res_B^A(A)$ in
$\B_B^\loc$. Thus, $L$ is Frobenius if and only if $\Res_B^A(A)$ is a
Frobenius algebra in $\B_B^\loc$.

By a similar argument, $A\in\B_B$ is Frobenius iff $\Ind_B^A$ is a Frobenius functor. As $\B_B^\loc$ is a tensor full subcategory of $\B_B$ and $\Res_B^A(A)\in\B_B^\loc$, it follows that $A$ is a Frobenius algebra in $\B_B$ iff it is a Frobenius algebra in $\B_B^\loc$. Hence, $\Ind_B^A$ is Frobenius iff $L$ is Frobenius. 

On the other hand, note that $L$ is surjective. Thus, by \cite[Lemma~4.3 and Theorem~4.8]{shimizu2017relative}, it is Frobenius iff
\begin{equation}\label{eq:Frob-unimodular}
  L(D_{\B_B^\loc}) \cong D_{\beta(B)} .
\end{equation}
By assumption, $\B_B^\loc$ is unimodular. Thus, \eqref{eq:Frob-unimodular} is equivalent to the unimodularity of $\beta(B)$.
\end{proof}


\subsection{The Drinfeld center case}
Let $\I_\C:\C\to\Z(\C)$ be the right adjoint of the forgetful functor.
We define the canonical algebra of $\C$ by
\[
  \bA:=\I_\C(\unit_\C)\in\Z(\C).
\]
By \cite[Proposition~6.6]{shimizu2026commutative}, $\bA$ is a
commutative, indecomposable and exact algebra. It is therefore simple by
\cite[Theorem~7.1]{coulembier2025simple}. Moreover,
\begin{equation}\label{eq:canonical-algebra-module-categories}
  \Z(\C)_{\bA}\simeq\C,
  \qquad
  \Z(\C)_{\bA}^{\loc}\simeq\Vect.
\end{equation}
Here the first equivalence is
\cite[Proposition~6.6(c)]{shimizu2026commutative}, and the second is
\cite[Remark~6.10]{shimizu2026commutative}.

\begin{corollary}\label{cor:Z(C)-case}
  Let $\C$ be a finite tensor category. 
  \begin{enumerate}
    \item There is a lattice isomorphism
    $\Lalg{\bA}^{\mathrm{op}}\cong\Lten{\C}$.
    \item The poset of Frobenius extensions $B\subseteq \bA$ is anti-isomorphic to the poset of tensor subcategories of $\C$ that are unimodular.
  \end{enumerate}
\end{corollary}
\begin{proof}
The category $\Z(\C)$ is nondegenerate. By
\eqref{eq:canonical-algebra-module-categories}, the interval
$[\Z(\C)_{\bA}^{\loc},\Z(\C)_{\bA}]_{\otimes}$ is identified with
$\Lten{\C}$. Thus, part (1) follows from
Theorem~\ref{thm:nondegenerate-lattice-isomorphism}. Part (2) follows from
Theorem~\ref{thm:frob-extn-unimodular}.
\end{proof}

When $\C$ is a fusion category, part~(1) recovers
\cite[Theorem~4.10]{davydov2013witt}.

Let $\M$ be an indecomposable exact left $\C$-module category \cite[\S7.5]{etingof2016tensor}. We denote the dual tensor category as $\C^*_{\M} = \Fun_{\C}(\M,\M)^\rev$. To clarify, in $\C^*_\M$, $F\otimes G = G \circ F$.  Let $S_\M:\Z(\C)\xrightarrow{\ \simeq\ }\Z(\C_\M^*)$ be the Schauenburg equivalence \cite{schauenburg2001monoidal} (see also \cite[Theorem~3.13]{shimizu2020further}), and set
\[
  \Psi_\M:=\U_{\C_\M^*}\circ S_\M,
  \qquad
  \bA_\M:=\Psi_\M^{\ra}(\unit_{\C_\M^*})\in\Z(\C).
\]
We have the following consequence.

\begin{corollary}\label{cor:module-category-adjoint-algebra}
  \begin{enumerate} 
    \item If $\M$ is an indecomposable exact left $\C$-module category, then
    there is a lattice isomorphism
    $\Lalg{\bA_\M}^{\mathrm{op}}\cong\Lten{\C^*_\M}$. 
    \item If $\M$ is an indecomposable exact left $\C$-module category, then Frobenius extensions $B\subseteq \bA_\M$ are anti-isomorphic to tensor subcategories of $\C^*_\M$ that are unimodular. 
  \end{enumerate}
\end{corollary}
\begin{proof}
Under $S_\M$, the algebra $\bA_\M$ corresponds to the canonical
algebra of $\C_\M^*$. Thus, $S_\M$ identifies their subalgebras and
their Frobenius extensions. Both claims now follow by applying
Corollary~\ref{cor:Z(C)-case} to $\C_\M^*$.
\end{proof}


\subsection{Example: pointed braided fusion categories}

We now make the full Galois connection and its two fixed-point intervals
explicit in a pointed example. Let $\Gamma$ be a finite abelian group, written
additively, and let
$q:\Gamma\to\kk^\times$ be a quadratic form. The corresponding pointed
braided fusion category is denoted by $\B=\C(\Gamma,q)$; see
\cite[\S2.11]{drinfeld2010braided}.

Write
\[
 b_q(g,x)=\frac{q(g+x)}{q(g)q(x)}
\]
for the associated symmetric bicharacter. For $S\subseteq\Gamma$, set
\[
 S^\perp=\{g\in\Gamma\mid b_q(g,s)=1\text{ for all }s\in S\},
 \qquad
 R=\Gamma^\perp=\operatorname{rad}(b_q).
\]
Denote by $\delta_g$ the simple object indexed by $g\in\Gamma$. The double
braiding between $\delta_g$ and $\delta_x$ is the scalar $b_q(g,x)$. Hence the
M\"uger center of $\B$ is $\C(R,q|_R)$.

Fix a subgroup $H\leq\Gamma$ such that $q|_H=1$. Then $H\leq H^\perp$, and
the object
\[
 A_H=\bigoplus_{h\in H}\delta_h
\]
admits a commutative algebra structure. After choosing an abelian
$3$-cocycle representing $q$, its multiplication is determined by a normalized
$2$-cochain on $H$. We fix one such algebra structure; see
\cite[Lemma~2.9]{shimizu2026ribbon}.

For each subgroup $L\leq H$, let $A_L\subseteq A_H$ be the subalgebra obtained
by restricting the multiplication. Since $A_H$ is multiplicity-free and its
homogeneous multiplication maps are non-zero, every subalgebra of $A_H$ is
of this form for a unique $L\leq H$. Moreover,
\[
 A_H\cap\B'=A_{H\cap R}.
\]
Thus the full lattice $\Lalg{A_H}$ corresponds to all subgroups of $H$,
whereas the distinguished interval $[A_{H\cap R},A_H]_{\mathrm{alg}}$
corresponds to the subgroup interval $[H\cap R,H]$.

We next describe the tensor-subcategory side. For $g\in\Gamma$, set $M_{g+H}=\delta_g\otimes A_H$. 
The simple-current module classification
\cite[Theorems~4.11 and~4.12]{shimizu2026ribbon} shows that $\B_{A_H}$ is
pointed, with simple objects $M_{g+H}$ indexed by $\Gamma/H$, and
\[
 M_{g+H}\otimes_{A_H}M_{x+H}\cong M_{g+x+H}.
\]
In particular, $A_H=M_H$ is simple. By
\cite[Proposition~4.21]{shimizu2026ribbon}, the module $M_{g+H}$ is local if
and only if $g\in H^\perp$. Therefore the quadratic form descends to
\[
 \overline q:H^\perp/H\to\kk^\times,
 \qquad
 \overline q(g+H)=q(g),
\]
and there is a braided equivalence $\B_{A_H}^{\loc}\simeq\C(H^\perp/H,\overline q)$.  In the nondegenerate characteristic-zero case, this also appears in
\cite[Proposition~5.17]{davydov2013witt}.

For $H\leq K\leq\Gamma$, let
\[
 \E_K=\langle M_{g+H}\mid g\in K\rangle
 \subseteq \B_{A_H}.
\]

\begin{proposition}\label{prop:pointed-galois-connection}
The assignments $L\mapsto A_L$ and $K\mapsto\E_K$ give identifications
\[
\begin{aligned}
 \Lalg{A_H}
 \cong\{L\mid L\leq H\}\qquad\quad \,\; ,
 & \quad
 [A_{H\cap R},A_H]_{\mathrm{alg}}
 \cong\{L\mid H\cap R\leq L\leq H\},\\
 \Lten{\B_{A_H}}
 \cong\{K\mid H\leq K\leq\Gamma\} \quad,
 & \quad
 \;\;\,[\B_{A_H}^{\loc},\B_{A_H}]_{\otimes}
 \cong\{K\mid H^\perp\leq K\leq\Gamma\}.
\end{aligned}
\]
Under these identifications, the full Galois connection is given by
\begin{equation}\label{eq:pointed-galois-maps}
 \beta(A_L)=\E_{L^\perp},
 \qquad
 \alpha(\E_K)=A_{H\cap K^\perp}.
\end{equation}
Its closure operators are
\begin{equation}\label{eq:pointed-galois-closures}
 \alpha\beta(A_L)=A_{L+(H\cap R)},
 \qquad
 \beta\alpha(\E_K)=\E_{K+H^\perp}.
\end{equation}
\end{proposition}

\begin{proof}
Tensor subcategories of a pointed fusion category are determined by subgroups
of the group of simple objects. Since this group is $\Gamma/H$, all tensor
subcategories of $\B_{A_H}$ are the categories $\E_K$ with
$H\leq K\leq\Gamma$. The local subcategory is $\E_{H^\perp}$, so the closed
interval consists of those $\E_K$ with $H^\perp\leq K\leq\Gamma$. Together
with the description of the subalgebras above, this proves the four displayed
identifications.

We compute $\beta$ before restricting to
$[A_{H\cap R},A_H]_{\mathrm{alg}}$. For every $L\leq H$, the
simple local $A_L$-modules are $\delta_g\otimes A_L$ with $g\in L^\perp$, and
induction gives
\[
 (\delta_g\otimes A_L)\otimes_{A_L}A_H
 \cong \delta_g\otimes A_H=M_{g+H}.
\]
Since $H\leq L^\perp$, the image of induction is the pointed tensor
subcategory indexed by $L^\perp/H$. Therefore $\beta(A_L)=\E_{L^\perp}$ 
for every $L\leq H$.

Let $H\leq K\leq\Gamma$. For every $L\leq H$,
Theorem~\ref{thm:general-galois-connection} gives
\[
 A_L\subseteq\alpha(\E_K)
 \iff \E_K\subseteq\beta(A_L)
 \iff K\leq L^\perp
 \iff L\leq K^\perp.
\]
The largest subgroup $L\leq H$ satisfying the last condition is
$H\cap K^\perp$. Hence $\alpha(\E_K)=A_{H\cap K^\perp}$, which proves
\eqref{eq:pointed-galois-maps} on the full lattices.

The subalgebra generated by $A_L$ and $A_{H\cap R}$ is
$A_{L+(H\cap R)}$, while the tensor subcategory generated by $\E_K$ and
$\E_{H^\perp}$ is $\E_{K+H^\perp}$. The closure formulas in
Theorem~\ref{thm:general-galois-connection} now give
\eqref{eq:pointed-galois-closures}.
\end{proof}

\begin{remark}\label{rem:pointed-positive-characteristic}
No characteristic-zero assumption is used in
Proposition~\ref{prop:pointed-galois-connection}. If
$\operatorname{char}(\kk)=p>0$, then the $p$-primary subgroup of $\Gamma$ is
contained in $R$, since $\kk^\times$ has no non-trivial elements of
$p$-power order. Thus nondegeneracy forces $p\nmid|\Gamma|$, but the formulas
\eqref{eq:pointed-galois-maps} and \eqref{eq:pointed-galois-closures} remain
valid when $q$ is degenerate and $p$ divides $|\Gamma|$. 
\end{remark}

\begin{corollary}
The maps between $\Lalg{A_H}$ and
$[\B_{A_H}^{\loc},\B_{A_H}]_{\otimes}$ are mutually inverse if and only if
$H\cap R=0$. The maps on the two full lattices are mutually inverse if and
only if, in addition, $H^\perp=H$. If $q$ is nondegenerate, the first
correspondence is given by the mutually inverse maps $L\mapsto L^\perp$ and
$K\mapsto K^\perp$.
\end{corollary}

\begin{proof}
The first two claims follow from \eqref{eq:pointed-galois-closures}. If $q$ is
nondegenerate, then $R=0$ and $K^\perp\leq H$ whenever
$H^\perp\leq K$. Hence $H\cap K^\perp=K^\perp$, and
\eqref{eq:pointed-galois-maps} becomes the stated pair of maps.
\end{proof}


\section{Hopf ideals and normal coideal subalgebras}\label{sec:hopf-ideals}


Let $\mathcal{B}$ be a braided tensor category, and let $H$ be a Hopf algebra in $\mathcal{B}$ with multiplication $m_H : H \otimes H \to H$, unit $u_H : \unit \to H$, comultiplication $\Delta_H : H \to H \otimes H$, counit $\varepsilon_H : H \to \unit$ and antipode $S_H : H \to H$. To state the main result of this section, we recall some terminology from the Hopf algebra theory. A {\em Hopf ideal} of $H$ is a two-sided ideal $J$ of the algebra $H$ such that $\Delta_H(J) \subseteq H \otimes J + J \otimes H$\footnote{Here, the right hand side is the image of
$(H\otimes J)\oplus(J\otimes H)\longrightarrow H\otimes H$ induced by the two inclusions.}, $\varepsilon_H(J) = 0$ and $S_H(J) \subseteq J$. A {\em left coideal subalgebra} of $H$ is a subalgebra of $H$ that is simultaneously a left $H$-subcomodule of $H$, where $H$ is viewed as a left $H$-comodule by $\Delta_H$. A left coideal subalgebra of $H$ is said to be {\em normal} if it is closed under the left adjoint action $\triangleright_{\mathrm{ad}}$ of $H$ (see \eqref{eq:braided-adjoint-action} for the definition). In this section, we prove:

\begin{theorem}
  \label{thm:CSA-Hopf-ideal-bijection}
  Let $\mathcal{B}$ be a braided finite tensor category, and let $H$ be a Hopf algebra in $\mathcal{B}$. Then there are well-defined order-preserving bijections
  \begin{equation*}
    \begin{tikzcd}
      \{ \text{Hopf ideals of $H$} \}
      \arrow[r, shift left=1.2, "{\mathcal{F}}"]
      & \{ \text{normal left coideal subalgebras of $H$} \}
      \arrow[l, shift left=1.2, "{\mathcal{G}}"]
    \end{tikzcd}
  \end{equation*}
  given by $\mathcal{F}(J) = H^{\mathrm{co}(H/J)}$ and $\mathcal{G}(K) = H K^{+}$.
\end{theorem}

The notation appearing in the definitions of $\mathcal{F}$ and $\mathcal{G}$ is a natural generalization to our setting of the standard notation used in Hopf algebra theory. For further details, see \S\ref{subsec:braided-Hopf-Gal-conn}.

For an ordinary Hopf algebra $H$, Takeuchi established a bijection between certain classes of Hopf ideals of $H$ and certain classes of normal coideal subalgebras $H$ \cite{takeuchi1994quotient}. 
For $\B = \Vect$ (the category of finite-dimensional vector spaces over $\kk$), the bijection of Theorem \ref{thm:CSA-Hopf-ideal-bijection} is obtained from Takeuchi's correspondence and Skryabin's results on faithful (co)flatness over coideal subalgebras \cite{skryabin2007projectivity}. Angiono and Campagnolo have established the bijection for the case where $\B$ is the category of Yetter-Drinfeld modules over a finite-dimensional Hopf algebra \cite[Proposition 2.1]{angiono2025posets}. 
Our results require the strong assumption of finiteness, but within the finite setting, it substantially generalizes known results.

Before we start the proof, we note the following two questions:

\begin{question}
    After suitable modifications, can we extend Theorem \ref{thm:CSA-Hopf-ideal-bijection} to the case where $\B$ is an ind-completion of a braided tensor category? An affirmative answer recovers Takeuchi's correspondence as the case where $\B = \Vect$.
\end{question}

\begin{question}
    Takeuchi’s correspondence also gives a bijection between the set of (not necessarily normal) coideal subalgebras of $H$ satisfying a certain condition and the set of $H$-module factor coalgebras of $H$ satisfying a certain condition; see, {\it e.g.}, \cite{skryabin2025takeuchi} for an exposition. Can we interpret, formulate and verify the correspondence in the setting of tensor categories?
\end{question}

\subsection{Modules and comodules over a Hopf algebra}

We first fix the notation. Let $H$ be a Hopf algebra in a braided monoidal category $\B$ with braiding $c$.
Given a left (right) $H$-module $M$, we denote by $a^l_M$ ($a^r_M$) the left (right) action of $H$ on $M$.
Similarly, given a left (right) $H$-comodule $M$, we denote by $\delta^l_M$ ($\delta^r_M$) the left (right) coaction of $H$ on $M$. We follow the convention of \cite{etingof2016tensor} for dual objects in a monoidal category. Thus, a left dual object of an object $X$ of a monoidal category is an object $X^*$ coming equipped with the evaluation $\mathrm{ev}_X : X^* \otimes X \to \unit$ and the coevaluation $\mathrm{coev}_X : \unit \to X \otimes X^*$.

There are four monoidal categories ${}_H\B$, $\B_H$, ${}^H\B$ and $\B^H$ of left modules, right modules, left comodules and right comodules over $H$, respectively. For example, if $X, Y \in {}_H\B$, then the tensor product $X \otimes Y$ is also a left $H$-module by the action given by
\begin{equation}
\label{eq:braided-Hopf-tensor-module-action}
    a^l_{X \otimes Y}
    = (a^l_X \otimes a^l_Y) (\id_H \otimes c_{H,X} \otimes \id_Y) (\Delta_H \otimes \id_X \otimes \id_Y).
\end{equation}

If the underlying object of $X \in {}_H\B$ has a left dual object in $\B$, then $X^*$ is in fact a left dual object in ${}_H\B$ by the action given by
\begin{equation}
\label{eq:braided-Hopf-dual-module-action}
    a^l_{X^*} =
    (\mathrm{ev}_X \otimes \id_{X^*}) (\id_{X^*} \otimes a^l_X \otimes \id_{X^*})
    (c_{H,X^*} \otimes \mathrm{coev}_X) (S_H \otimes \id_{X^*}).
\end{equation}
Provided that the underlying object of $X$ has a right dual and $S_H$ is invertible, a right dual object of $X$ in ${}_H\B$ is constructed based on ${}^*\!X \in \B$ in a similar manner, but by using the inverse of $S_H$ instead of $S_H$.

\subsection{Yetter-Drinfeld modules}

Next, we recall the notion of Yetter-Drinfeld module, which plays a crucial role in the proof of Theorem \ref{thm:CSA-Hopf-ideal-bijection}.
Let $\mathcal{B}$ be a braided monoidal category with braiding $c$, and let $H$ be a Hopf algebra in $\mathcal{B}$ with structure morphisms denoted as above.

\begin{definition}[\cite{MR1492897}]
  A (left-left) {\em Yetter-Drinfeld module over $H$} is a left $H$-module $M$ endowed with a structure of a left $H$-comodule such that the Yetter-Drinfeld condition
\begin{equation}
  \label{eq:braided-YD-condition-0}
  \begin{aligned}
    & (m_H \otimes a^l_M) (\id_H \otimes c_{H,H} \otimes \id_M) (\Delta_H \otimes \delta^l_M) \\
    & \quad = (m_H \otimes \id_M) (\id_H \otimes c_{M,H})
    (\delta^l_M a^l_M \otimes \id_H) (\id_H \otimes c_{H,M})(\Delta_H \otimes \id_M)
  \end{aligned}
\end{equation}
is satisfied. We denote by ${}^H_H\mathcal{YD}(\mathcal{B})$ the category whose objects are Yetter-Drinfeld modules over $H$ and whose morphisms are $H$-linear $H$-colinear morphisms in $\mathcal{B}$. The category ${}^H_H\mathcal{YD}(\mathcal{B})$ is called the Yetter-Drinfeld category of $H$.
\end{definition}

To shorten expressions, we define $\Xi_{P,Q,X}:P\otimes Q\otimes X \to Q\otimes P\otimes X$ as
\begin{align*}
  \Xi_{P,Q,X}
  & = (a^{lr}_Q \otimes \id_P \otimes \id_X)
  \circ (\id_H \otimes \id_Q \otimes c_{P \otimes X, H}) \\
  & \quad \circ (\id_H \otimes c_{P,Q} \otimes \id_X \otimes S_H) \\
  & \quad \circ (\id_H \otimes \id_P \otimes c_{H, Q \otimes X})
    \circ (\delta^{lr}_{P} \otimes \id_Q \otimes \id_X)
\end{align*}
for an $H$-bicomodule $P$, an $H$-bimodule $Q$ and $X \in \B$, where
\begin{equation*}
  a^{lr}_{Q} = a^l_{Q} (\id_H \otimes a^r_Q)
  = a^r_Q(a^l_Q \otimes \id_H) : H \otimes Q \otimes H \to Q
\end{equation*}
is the `two-sided action' on $Q$ and $\delta^{lr}_P : P \to H \otimes P \otimes H$ is the `two-sided coaction' on $P$ defined in a dual way as $a^{lr}_P$. After a slight modification of \cite[Lemma 2.2]{laugwitz2019comodalg} using the naturality of the braiding, we find that the Yetter-Drinfeld condition \eqref{eq:braided-YD-condition-0} is equivalent to
\begin{equation}
    \label{eq:braided-YD-condition}
    \delta^l_M \circ a^l_M = (\id_H \otimes a^l_M) \circ \Xi_{H,H,M} \circ (\id_H \otimes \delta^l_M),
\end{equation}
where $H$ is viewed as an $H$-bi(co)module by the (co)multiplication of $H$.

The category ${}^H_H\mathcal{YD}(\mathcal{B})$ has a monoidal structure inherited from the monoidal categories ${}_H\mathcal{B}$ and ${}^H\mathcal{B}$.
Moreover, it has a lax braiding $\Sigma$ given by
\begin{equation}
  \label{eq:YD-braiding}
  \Sigma_{M,N} = (a^l_N \otimes \id_M) (\id_H \otimes c_{M,N}) (\delta^l_M \otimes \id_N)
\end{equation}
for $M, N \in {}^H_H\mathcal{YD}(\mathcal{B})$. If the antipode $S_H$ is invertible, then the lax braiding $\Sigma$ is invertible \cite[theorem 4.1.1]{MR1492897}, and thus it is a braiding.

We can realize ${}^H_H\mathcal{YD}(\mathcal{B})$ as the category of comodules over a comonad on ${}_H\mathcal{B}$ in a similar way as Schauenburg did for quasi-Hopf algebras in \cite{MR1897403}. To be precise, for $M \in {}_H\mathcal{B}_H$ and $X \in {}_H\mathcal{B}$, we define $M \triangleright X \in {}_H\mathcal{B}$ as the object $M \otimes X$ endowed with the left $H$-action
\begin{equation*}
  a^{l}_{M \triangleright X} := (\id_M \otimes a^{l}_X) \circ \Xi_{H,M,X} : H \otimes M \otimes X \to M \otimes X.
\end{equation*}
The category ${}_H\mathcal{B}_H$ has a monoidal structure inherited from ${}_H\mathcal{B}$ and $\mathcal{B}_H$. 
One can verify that the operation $\triangleright$ makes ${}_H\mathcal{B}$ a left module category over the monoidal category ${}_H\mathcal{B}_H$ with the module associator
\begin{equation*}
  \id_{M \otimes N \otimes X} : M \triangleright (N \triangleright X)
  \to (M \otimes N) \triangleright X
  \quad (M, N \in {}_H\mathcal{B}_H, X \in {}_H\mathcal{B}).
\end{equation*}
Hence each coalgebra $C$ in ${}_H\mathcal{B}_H$ defines a comonad $C \triangleright (-)$ on ${}_H\mathcal{B}$. It follows from the definition of a Hopf algebra that $H$ is a coalgebra in ${}_H\mathcal{B}_H$. Thus we have the comonad $Z := H \triangleright (-)$ on ${}_H\mathcal{B}$.

\begin{lemma}
  The category ${}^H_H\mathcal{YD}(\mathcal{B})$ is identified with the category of comodules over the comonad $Z$ in a way compatible with the forgetful functor to ${}_H\mathcal{B}$.
\end{lemma}
\begin{proof}
  We fix a left $H$-module $M$ in $\mathcal{B}$. By the definition of the comonad $Z$, we see that a morphism $\delta : M \to H \otimes M$ in $\mathcal{B}$ makes $M$ a $Z$-comodule if and only if the following three equations are satisfied:
  \begin{gather*}
    (\Delta_H \otimes \id_M) \circ \delta
    = (\id_H \otimes \delta) \circ \delta,
    \quad (\varepsilon_H \otimes \id_M) \circ \delta = \id_M, \\
    \delta \circ a^l_M = a_{H \triangleright M}^l \circ (\id_H \otimes \delta).
  \end{gather*}
  The first two equations are equivalent to that $M$ is a left $H$-comodule in $\mathcal{B}$ with coaction $\delta$. Moreover, the third equation is equivalent to the Yetter-Drinfeld condition \eqref{eq:braided-YD-condition} with $\delta_M^{l} = \delta$. From this observation, we find that the data defining a $Z$-comodule is exactly the same as that defining a Yetter-Drinfeld module over $H$.
\end{proof}

An immediate consequence of the above lemma is:

\begin{lemma}\label{lem:braided-YD-adjoint}
  The cofree $Z$-comodule functor
  \begin{equation*}
    R : {}_H\mathcal{B} \to {}^H_H\mathcal{YD}(\mathcal{B}),
    \quad X \mapsto Z(X)
    \quad (X \in {}_H\mathcal{B})
  \end{equation*}
  is right adjoint to the forgetful functor $U : {}^H_H\mathcal{YD}(\mathcal{B}) \to {}_H\mathcal{B}$.
  The unit and the counit of this adjunction are given respectively by
  \begin{equation*}
    \delta^l_M : M \to R U(M)
    \quad (M \in {}^H_H\mathcal{YD}(\mathcal{B})),
    \quad \varepsilon_H \otimes \id_V : U R(V) \to V
    \quad (V \in {}_H\mathcal{B}).
  \end{equation*}
\end{lemma}

Since $U$ is a strong monoidal functor (in fact, it is a strict monoidal functor), its right adjoint $R$ is a lax monoidal functor such that the unit and the counit of the adjunction are monoidal natural transformations. The structure morphisms
\begin{equation*}
  R^{(2)}_{V,W} : R(V) \otimes R(W) \to R(V \otimes W)
  \quad \text{and}
  \quad R^{(0)} : \unit \to R(\unit)
\end{equation*}
of $R$ are given as follows:

\begin{lemma}
  For $V, W \in {}_H\mathcal{B}$, we have
  \begin{equation*}
    R^{(2)}_{V,W} = (m_H \otimes \id_V \otimes \id_W) (\id_H \otimes c_{V,H} \otimes \id_W),
    \quad R^{(0)} = u_H.
  \end{equation*}
\end{lemma}
\begin{proof}
  Let $\tilde{\eta}$ and $\tilde{\varepsilon}$ denote the unit and the counit of the adjunction $U \dashv R$ given by the previous lemma, respectively. It is known that $R^{(2)}_{V,W}$ and $R^{(0)}$ are given by
  \begin{equation*}
    R^{(2)}_{V,W} = R(\tilde{\varepsilon}_V \otimes \tilde{\varepsilon}_W) \circ \tilde{\eta}_{R(V) \otimes R(W)}
    \quad \text{and} \quad
    R^{(0)} = \tilde{\eta}_{\unit},
  \end{equation*}
  respectively, where $U(R(V) \otimes R(W))$ is identified with $U R(V) \otimes UR(W)$ in the first equation thanks to the strictness of $U$. Now it is straightforward to verify the claim.
\end{proof}

\subsection{Modules and local modules over the adjoint algebra}
\label{subsec:braided-Hopf-adj-alg}

We retain the notation of the previous subsection. We also assume that the antipode $S_H$ is invertible. Then, as we have remarked, ${}^H_H\mathcal{YD}(\mathcal{B})$ is a braided monoidal category with the braiding $\Sigma$ given by \eqref{eq:YD-braiding}. We note that the morphism $\Sigma_{M,N}$ is naturally defined for all $M \in {}^H\mathcal{B}$ and $N \in {}_H\mathcal{B}$, and the braided strong monoidal functor $\widetilde{U} : {}^H_H\mathcal{YD} \to \mathcal{Z}({}_H\mathcal{B})$ is also defined by $\widetilde{U}(M) = (M, \Sigma_{M,-})$. The functor $U$ is the composition of $\widetilde{U}$ and the forgetful functor $\mathcal{Z}({}_H\mathcal{B}) \to {}_H\mathcal{B}$.
Namely, $U$ is {\em central} in the sense of \cite[Definition 8.8.6]{etingof2016tensor}.

Since $R$ is lax monoidal, $\mathbf{A} := R(\unit)$ is an algebra in ${}^H_H\mathcal{YD}(\mathcal{B})$. We call $\mathbf{A}$ the {\em adjoint algebra}.
Since $U$ is central, \cite[Proposition 8.8.8]{etingof2016tensor} shows that $\mathbf{A}$ is in fact a commutative algebra in ${}^H_H\mathcal{YD}(\mathcal{B})$.
By construction, $\mathbf{A} = H$ as an object of $\mathcal{B}$. The left $H$-action, the left $H$-coaction, the multiplication and the unit of $\mathbf{A}$ are given by $\triangleright_{\mathrm{ad}}$, $\Delta_H$, $m_H$ and $u_H$, respectively, where
\begin{equation}
  \label{eq:braided-adjoint-action}
  \begin{gathered}
  \triangleright_{\mathrm{ad}}
  = (\id_H \otimes \varepsilon_H) \Xi_{H,H,\unit}
  = m_H^{(3)} (\id_H \otimes \id_H \otimes S_H) (\id_H \otimes c_{H,H}) (\Delta_H \otimes \id_H) \\
  (m_H^{(3)} = m_H(m_H \otimes \id_H) = m_H (\id_H \otimes m_H))
  \end{gathered}
\end{equation}
is the {\em adjoint action}. Thus we have:

\begin{lemma}
  A normal left coideal subalgebra of $H$ is the same thing as a subalgebra of the adjoint algebra $\mathbf{A}$ in ${}^H_H\mathcal{YD}(\mathcal{B})$.
\end{lemma}

Roughly speaking, Theorem \ref{thm:CSA-Hopf-ideal-bijection} will be obtained by applying Theorem \ref{thm:general-galois-connection} to the adjoint algebra $\mathbf{A}$. Thus, we now start working on the category of $\mathbf{A}$-modules and the category of local $\mathbf{A}$-modules. Given $V \in {}_H\B$, we denote by $\widetilde{R}(V)$ the object $R(V)$ viewed as a right $\mathbf{A}$-module by the lax monoidal structure of $R$. Then we have a functor
\begin{equation*}
  \widetilde{R} : {}_H\mathcal{B} \to ({}^H_H\mathcal{YD}(\mathcal{B}))_{\mathbf{A}}.
\end{equation*}
When $\B$ is a finite tensor category, ${}^H_H\mathcal{YD}(\mathcal{B})$ is also a braided finite tensor category and $U$ is a surjective tensor functor \cite[Section 6]{shimizu2019non}. Moreover, by \cite[Corollary 6.7]{shimizu2026commutative}, the functor $\widetilde{R}$ is in fact an equivalence of finite tensor categories. 
In view of this observation, we examine when $\widetilde{R}(V)$ is a local module over $\mathbf{A}$ (without assuming that $\B$ is finite).
We discuss this problem in a more general setting: Let $K$ be a normal left coideal subalgebra of $H$. We note that $K$ is a commutative algebra in ${}^H_H\mathcal{YD}(\mathcal{B})$ as a subalgebra of $\mathbf{A}$. Given a right $\mathbf{A}$-module $M$, we denote by $M|_K$ the right $K$-module with action given by $a^r_M(\id_M \otimes i_K)$, where $i_K : K \to \mathbf{A}$ is the inclusion morphism. We now prove:

\begin{lemma}
  \label{lem:local-condition-for-RVK}
  Let $V \in {}_H\mathcal{B}$.
  Then $\widetilde{R}(V)|_K$ is a local $K$-module if and only if
  \begin{equation}
    \label{eq:local-condition-for-RVK}
    a^{l}_V(i_K \otimes \id_V) = \varepsilon_H i_K \otimes \id_V.
  \end{equation}
\end{lemma}

\begin{proof}
  We apply Lemma~\ref{lem:restriction-locality} to the central functor $U : {}^H_H\mathcal{YD}(\mathcal{B}) \to {}_H\mathcal{B}$, its right adjoint $R$ given in Lemma~\ref{lem:braided-YD-adjoint}, and the subalgebra $K$.
  In this specialization, the half-braiding is given by $\Sigma_{K,V}$.
  We write $\widetilde\varepsilon$ for the counit of $U\dashv R$ given in Lemma \ref{lem:braided-YD-adjoint}.
  The morphism $q$ of Lemma~\ref{lem:restriction-locality} is then given by  $q=\widetilde\varepsilon_{\unit}\circ U(i_K) =\varepsilon_H i_K$.
  The underlying morphism of $\widetilde\varepsilon_V$ is split epic, with
  section $u_H\otimes\id_V$. Since the forgetful functor to $\mathcal B$
  is faithful, the same argument after tensoring with $K$ shows that
  $\id_K\otimes\widetilde\varepsilon_V$ is epic. Therefore
  \eqref{eq:restriction-locality-epic} gives
  \[
    \widetilde{R}(V)|_K\text{ is local}
    \quad\Longleftrightarrow\quad
    (\id_V\otimes\varepsilon_H i_K)\circ\Sigma_{K,V}
    =(\varepsilon_H i_K)\otimes\id_V.
  \]
  By \eqref{eq:YD-braiding} and the counit identity
  $(\id_H\otimes\varepsilon_H)\Delta_H=\id_H$, the left-hand side is
  $a_V^l(i_K\otimes\id_V)$. This gives
  \eqref{eq:local-condition-for-RVK}.
\end{proof}

\subsection{Hopf ideals and tensor full subcategories}

From now on, we assume that $\mathcal{B}$ is a braided tensor category.
Let $H$ be a Hopf algebra in $\mathcal{B}$ (we note that the rigidity of $\mathcal{B}$ implies that the antipode of $H$ is invertible \cite{MR1685417}). In this subsection, we establish a bijection between Hopf ideals of $H$ and the interval $[\B,{}_H\B]_{\otimes}$.

We first identify ideals of $H$ with a class of subcategories of ${}_H\B$. In fact, more generally, we have Lemma \ref{lem:ideal-subcat-bijection} below. Let $\C$ be a tensor category, and let $A$ be an algebra in $\C$. We note that a left $A$-module is the same thing as an object $X \in \C$ together with an algebra map $\rho_X : A \to X \otimes X^*$. The {\em annihilator} of $X$ is defined and denoted by $\mathrm{Ann}(X) = \Ker(\rho_X)$. Given an ideal $J$ of $A$, we identify ${}_{A/J}\C$ as a full subcategory of ${}_A\C$ as
\begin{equation*}
  {}_{A/J}\mathcal{C} = \{ X \in {}_A\mathcal{C} \mid J \subseteq \mathrm{Ann}(X) \}.
\end{equation*}
The category ${}_A\C$ is a right $\C$-module category. By a {\em $\C$-module full subcategory} of ${}_A\C$, we mean a non-empty full subcategory of ${}_A\C$ that is closed under subquotients, finite direct sums and the action of $\mathcal{C}$. It is easy to see that ${}_{A/J}\C$ is a $\C$-module full subcategory.

\begin{lemma}
\label{lem:ideal-subcat-bijection}
    There is a bijection
    \begin{equation*}
    \{ \text{ideals of $A$} \}
    \to \{ \text{$\C$-module full subcategories of ${}_{A}\mathcal{C}$} \},
    \quad J \mapsto {}_{A/J}\C.
    \end{equation*}
\end{lemma}
\begin{proof}
The map is injective since an ideal $J$ of $A$ is recovered from $\mathcal{S} = {}_{A/J}\C$ as the intersection of all annihilators of objects of $\mathcal{S}$. To prove the surjectivity, we note that an ideal of $A$ is nothing but a subobject of $A$ in ${}_A\C_A$. Let $\mathcal{E}$ be the category of $\kk$-linear right exact $\C$-module endofunctors on ${}_A\C$. The Eilenberg-Watts theorem gives an equivalence $\mathcal{E} \approx {}_A\C_A$ of categories. Since ${}_A\C_A$ is an abelian category, we have bijections
\begin{equation*}
    \{ \text{ideals of $A$} \}
    \approx \{ \text{quotients of $A$ in ${}_A\C_A$} \}
    \approx \{ \text{quotients of the identity functor in $\E$} \},
\end{equation*}
which send an ideal $J$ of $A$ to the object $(A/J) \otimes_A (-)$ of $\E$.

Now let $\mathcal{S}$ be a $\C$-module full subcategory of ${}_A\C$. By local finiteness of $\C$, the inclusion functor $i : \mathcal{S} \to {}_A\C$ has a right adjoint $i^{\mathrm{ra}}$ given by taking the maximal subobject belonging to $\mathcal{S}$. The dual argument shows that $i$ has a left adjoint $i^{\mathrm{la}}$ given by taking the maximal quotient object belonging to $\mathcal{S}$. In other words, $\mathcal{S}$ is a reflective subcategory of ${}_A\C$. Thus $t_{\mathcal{S}} := i i^{\mathrm{la}}$ is an idempotent monad on ${}_A\mathcal{C}$ whose category of modules is identified with $\mathcal{S}$. The functor $t_{\mathcal{S}}$ is $\kk$-linear and right exact as it has a right adjoint $i i^{\mathrm{ra}}$.

We consider the unit $\pi : \id_{{}_A\mathcal{C}} \to t_{\mathcal{S}}$ of the adjunction $i^{\mathrm{la}} \dashv i$. By definition, $\pi_M : M \to t_{\mathcal{S}}(M)$ is the quotient morphism for all $M \in {}_A\mathcal{C}$. Since $i : \mathcal{S} \to {}_A\mathcal{C}$ is a strong $\mathcal{C}$-module functor, $i^{\mathrm{la}}$ is an oplax $\mathcal{C}$-module functor such that $\pi$ is a morphism of oplax $\mathcal{C}$-module functors. The rigidity of $\mathcal{C}$ implies that $i^{\mathrm{la}}$ is in fact a strong $\mathcal{C}$-module functor. Therefore $\pi : \id_{{}_A\mathcal{C}} \to t_{\mathcal{S}}$ is in fact an epimorphism in the category $\mathcal{E}$.

Thus, by the above argument, $t_{\mathcal{S}}$ is isomorphic to $(A/J) \otimes_A (-)$ for some ideal $J$ of $A$. Hence $\mathcal{S}$ is identified with ${}_{A/J}\C$. The proof is done.
\end{proof}

Tensor-categorical conditions on ${}_{H/J}\B$ are interpreted as follows:

\begin{lemma}
\label{lem:braided-Hopf-ideals-1}
For an ideal $J$ of the Hopf algebra $H$, we have the following.
\begin{enumerate}
\item ${}_{H/J}\B$ contains the unit object of ${}_{H}\B$ if and only if $\varepsilon_H(J) = 0$.
\item ${}_{H/J}\B$ is closed under the tensor product if and only if $\Delta_H(J) \subseteq H \otimes J + J \otimes H$.
\item ${}_{H/J}\B$ is closed under duals if and only if $S_H^{\pm1}(J) \subseteq J$.
\end{enumerate}
\end{lemma}
\begin{proof}
Part (1) is obvious. To show Part (2), we note that the equation
\begin{equation*}
  \rho_{X \otimes Y} = (\id_X \otimes c_{Y \otimes Y^*, X^*}^{-1})(\rho_X \otimes \rho_Y) \Delta_H
\end{equation*}
holds for $X, Y \in {}_H\B$ by the definition \eqref{eq:braided-Hopf-tensor-module-action} of the action of $H$ on the tensor product. This equation proves the `if' part. The equation also implies
\begin{equation*}
  \Delta_H(\mathrm{Ann}(X \otimes Y)) \subseteq \Ker(\rho_X \otimes \rho_Y) = \mathrm{Ann}(X) \otimes H + H \otimes \mathrm{Ann}(Y).
\end{equation*}
The `only if' part is now proved as follows: We assume that ${}_{H/J}\B$ is closed under the tensor products. Then $T := H/J \otimes H/J$ belongs to ${}_{H/J}\mathcal{B}$. Thus $J \subseteq \mathrm{Ann}(T)$, and therefore
\begin{equation*}
    \Delta_H(J) \subseteq \Delta_H(\mathrm{Ann}(T))
    \subseteq \mathrm{Ann}(H/J) \otimes H + H \otimes \mathrm{Ann}(H/J)
    = J \otimes H + H \otimes J.
\end{equation*}

It remains to verify Part (3). For a left $H$-module $X$, we have
\begin{equation*}
    S_H(\mathrm{Ann}(X^*)) \subseteq \mathrm{Ann}(X), \quad
    S_H^{-1}(\mathrm{Ann}({}^*\!X)) \subseteq \mathrm{Ann}(X).
\end{equation*}
Indeed, the former follows from the definition \eqref{eq:braided-Hopf-dual-module-action} of the action of $H$ on $X^*$ and the latter is obtained in a similar way. The `if' part of Part (3) is now easy. The `only if' part is proved by letting $X = H/J$ in the above.
\end{proof}

We regard $\B$ as a full subcategory of ${}_H\B$, as $\B = \{ X \in {}_H\B \mid a^l_X = \varepsilon_H \otimes \id_X \}$. Then:
\begin{lemma}
  \label{lem:Hopf-ideals-and-tensor-subcates}
  There is a bijection
  \begin{equation*}
    \{ \text{Hopf ideals of $H$} \}
    \to [\mathcal{B},{}_H\mathcal{B}]_{\otimes},
    \quad J \mapsto {}_{H/J}\B.
  \end{equation*}
\end{lemma}
\begin{proof}
A tensor full subcategory of ${}_H\B$ containing $\B$ is a $\B$-module full subcategory of ${}_H\B$, and thus it is of the form ${}_{H/J}\B$ for some ideal $J$ of $H$. Since $S_H$ is invertible, and since $J$ is of finite length, the conditions $S_H(J) \subseteq J$ and $S_H^{-1}(J) \subseteq J$ are equivalent (in fact, they are equivalent to $S_H(J) = J$). The claim of this lemma is now a direct consequence of Lemma \ref{lem:braided-Hopf-ideals-1}.
\end{proof}

Finally, we provide a useful criterion for an ideal of $H$ to be a Hopf ideal.
A {\em bi-ideal} of a bialgebra $B$ is an ideal $J$ of $B$ such that $\Delta_B(J) \subseteq B \otimes J + J \otimes B$ and $\varepsilon_B(J) = 0$. In general, a bi-ideal of a Hopf algebra may not be closed under the antipode. Nevertheless, under our assumption that $\mathcal{B}$ is a braided tensor category, we have:

\begin{lemma}
  \label{lem:braided-Hopf-ideals-2}
  A bi-ideal of a Hopf algebra $H$ in $\B$ is a Hopf ideal.
\end{lemma}
\begin{proof}
  Let $D$ be a Hopf algebra in $\mathcal{B}$. Then $\mathcal{E} := \Hom_{\mathcal{B}}(D, D)$ is an algebra with respect to the convolution product $f \star g = m_D (f \otimes g) \Delta_D$.
  By definition, the antipode $S_D$ is the inverse of $\id_D$ in the algebra $\mathcal{E}$. Since $\mathcal{E}$ is finite-dimensional, we use the Cayley-Hamilton theorem to show that $S_D$ is written as a polynomial of $\id_D$ in $\mathcal{E}$, say $S_D = \sum_{j = 0}^m c_j \id_D^{\star j}$ for some $c_j \in \Bbbk$, where $(-)^{\star j}$ is the $j$-th power with respect to the convolution product. From this expression, it is obvious that $S_D(D') \subseteq D'$ for any subbialgebra $D'$ of $D$.

  Now we apply the above argument to $D = H^*$ and $D' = (H/J)^*$. Then it follows that the subbialgebra $(H/J)^*$ of $H^*$ is stable under the antipode $S_H^* : H^* \to H^*$. This means that $S_H(J) \subseteq J$. The proof is done.
\end{proof}

Although the following observation is not needed for the proof of the main result of this section, we include it since it may be of independent interest:

\begin{proposition}
    Let $\C$ be a full subcategory of ${}_H\B$, where $\B$ is a braided tensor category and $H$ is a Hopf algebra in $\mathcal{B}$. We assume that $\C$ contains $\B$ and is closed under subquotients, finite direct sums and the tensor product of ${}_H\B$. Then $\C$ is closed under duals, and thus it is a tensor full subcategory of ${}_H\B$.
\end{proposition}
\begin{proof}
  By Lemma~\ref{lem:ideal-subcat-bijection}, there is an ideal
  $J\subseteq H$ such that $\C={}_{H/J}\B$. Since $\C$ contains $\B$, 
  it contains the tensor unit of
  ${}_H\B$. As $\C$ is closed under tensor products,
  Lemma~\ref{lem:braided-Hopf-ideals-1} proves that $J$ is a bi-ideal. By
  Lemma~\ref{lem:braided-Hopf-ideals-2}, $J$ is a Hopf ideal. 
  Since the antipode $S_H$ is invertible and $J$ has finite length, the
  inclusion $S_H(J)\subseteq J$ implies $S_H(J)=J=S_H^{-1}(J)$. 
  Lemma~\ref{lem:braided-Hopf-ideals-1}(3) now shows that
  ${}_{H/J}\B=\C$ is closed under duals. Hence $\C$ is a tensor full
  subcategory of ${}_H\B$.
\end{proof}

\subsection{Galois connection between Hopf ideals and normal coideal subalgebras}
\label{subsec:braided-Hopf-Gal-conn}

We show that the maps $\mathcal{F}$ and $\mathcal{G}$ in Theorem \ref{thm:CSA-Hopf-ideal-bijection} are well-defined and constitute a Galois connection without assuming the finiteness of $\mathcal{B}$. Given a Hopf ideal $J$ of $H$, we define
\begin{equation*}
  H^{\mathrm{co}(H/J)} = \Ker(
  (\id_H \otimes \pi)\Delta_H - \id_H \otimes \pi u_H
  : H \to H \otimes (H/J)),
\end{equation*}
where $\pi : H \to H/J$ is the quotient map.

\begin{lemma}
  \label{lem:Hopf-ideal-to-CSA}
  $K := H^{\mathrm{co}(H/J)}$ is a normal left coideal subalgebra of $H$.
\end{lemma}
\begin{proof}
  We set $f = (\id_H \otimes \pi)\Delta_H$ and $g = \id_H \otimes \pi u_H$ for simplicity of notation. It is easy to see that both $f$ and $g$ are algebra morphisms in $\mathcal{B}$. Thus $K$ is an algebra in $\mathcal{B}$ as an equalizer of two algebra morphisms in $\mathcal{B}$.
  Since $f$ and $g$ are also left $H$-comodule morphisms in $\mathcal{B}$, the object $K$ is also a subcomodule of $H$.

It remains to show the normality of $K$.
Let $i_K : K \to H$ be the inclusion morphism.
By the universal property of $K$ as an equalizer, it suffices to verify the following equation:
  \begin{equation}
    \label{eq:Hopf-ideal-to-CSA-proof-1}
    f \circ \triangleright_{\mathrm{ad}} \circ (\id_H \otimes i_K)
    = g \circ \triangleright_{\mathrm{ad}} \circ (\id_H \otimes i_K).
  \end{equation}
  We note that an ideal of $H$ is closed under the adjoint action $\triangleright_{\mathrm{ad}}$.
  We denote by $\overline{\triangleright}_{\mathrm{ad}}$ the action of $H$ on $H/J$ induced by $\triangleright_{\mathrm{ad}}$. Then we have
  \begin{equation}
    \label{eq:Hopf-ideal-to-CSA-proof-2}
    \overline{\triangleright}_{\mathrm{ad}} (\id_H \otimes \overline{u})
    = \pi \circ \triangleright_{\mathrm{ad}} \circ (\id_H \otimes u_H) =  \varepsilon_H \otimes \overline{u}
  \end{equation}
  by the axioms of Hopf algebras, where $\overline{u} := \pi u_H$ is the unit of the quotient algebra $H/J$.
  We recall that $H$ is a Yetter-Drinfeld module with action $\triangleright_{\mathrm{ad}}$ and coaction $\Delta_H$. Now we compute the left-hand side of \eqref{eq:Hopf-ideal-to-CSA-proof-1} as follows: $f \circ \triangleright_{\mathrm{ad}} \circ (\id_H \otimes i_K)$
  \begin{align*}
    {}^{\eqref{eq:braided-YD-condition}}
    & = (\id_H \otimes \pi) \circ (\id_H \otimes \triangleright_{\mathrm{ad}})
    \circ \Xi_{H,H,H} \circ (\id_H \otimes \Delta_H) \circ (\id_H \otimes i_K) \\
    & = (\id_H \otimes \overline{\triangleright}_{\mathrm{ad}})
    \circ (\id_H \otimes \id_H \otimes \pi)
    \circ \Xi_{H,H,H} \circ (\id_H \otimes \Delta_H i_K) \\
    {}^{\text{(nat.)}}
    & = (\id_H \otimes \overline{\triangleright}_{\mathrm{ad}})
    \circ \Xi_{H,H,H/J} \circ (\id_H \otimes \id_H \otimes \pi) \circ (\id_H \otimes \Delta_H i_K) \\
    {}^{\text{(eq.)}}
    & = (\id_H \otimes \overline{\triangleright}_{\mathrm{ad}})
    \circ \Xi_{H,H,H/J} \circ (\id_H \otimes i_K \otimes \overline{u}) \\
    {}^{\text{(nat.)}}
    & = (\id_H \otimes \overline{\triangleright}_{\mathrm{ad}})
    \circ (\id_H \otimes \id_H \otimes \overline{u})
    \circ \Xi_{H,H,\unit} \circ (\id_H \otimes i_K) \\
    {}^{\eqref{eq:Hopf-ideal-to-CSA-proof-2}}
    & = (\id_H \otimes \varepsilon_H \otimes \overline{u})
    \circ \Xi_{H,H,\unit} \circ (\id_H \otimes i_K) \\
    {}^{\eqref{eq:braided-adjoint-action}}
    & = (\id_H \otimes \overline{u}) \circ \triangleright_{\mathrm{ad}} \circ (\id_H \otimes i_K)
    = g \circ \triangleright_{\mathrm{ad}} \circ (\id_H \otimes i_K),
  \end{align*}
  where (nat.) means that the naturality of $\Xi_{H,H,X}$ in $X \in \B$ is used there and (eq.) follows from the equalizer property of $K$. The proof is done.
\end{proof}

Now let $K$ be a normal left coideal subalgebra of $H$. Given subobjects $X$ and $Y$ of $H$, we denote by $X Y$ the image of $X \otimes Y$ under the multiplication $H \otimes H \to H$. We define $K^{+}$ to be the kernel of the restriction of the counit $\varepsilon_H : H \to \unit$ to $K$.

\begin{lemma}
  \label{lem:CSA-to-Hopf-ideal}
  $J := H K^{+}$ is a Hopf ideal of $H$.
  Moreover, we have $K^{+} H = J$.
\end{lemma}
\begin{proof}
  We first show $H K^{+} = K^{+} H$.
  We consider the adjoint algebra $\mathbf{A} \in {}^H_H\mathcal{YD}(\mathcal{B})$.
  As we have remarked, the braiding $\Sigma_{\mathbf{A},M}$ is naturally defined for all $M \in {}_H\mathcal{B}$. Let $i_K^{+} : K^{+} \to H$ be the inclusion morphism. Then we have
  \begin{gather*}
    H K^{+}
    = \mathrm{Im}(m_H \Sigma_{\mathbf{A},\mathbf{A}}^{-1} \Sigma_{\mathbf{A},\mathbf{A}}(\id_H \otimes i_K^{+}))
    = \mathrm{Im}(m_H \Sigma_{\mathbf{A},\mathbf{A}}(\id_H \otimes i_K^{+})) \\
    = \mathrm{Im}(m_H (i_K^{+} \otimes \id_H) \Sigma_{\mathbf{A},K^{+}})
    = \mathrm{Im}(m_H (i_K^{+} \otimes \id_H)) = K^{+}H,
  \end{gather*}
  where the first equality follows from the definition of $H K^{+}$, the second from the commutativity of $\mathbf{A}$, the third from the naturality of $\Sigma_{\mathbf{A}, M}$ for $M \in {}_H\mathcal{B}$, and the fourth from the invertibility of $\Sigma_{\mathbf{A},K^{+}}$.

  The equation $H K^{+} = K^{+} H$ implies that $J$ is an ideal generated by $K^{+}$.
  As full subcategories of ${}_H\mathcal{B}$, we have
  \begin{align*}
    {}_{H/J}\mathcal{B}
    & = \{ X \in {}_H\mathcal{B} \mid J \subseteq \mathrm{Ann}(X) \}
      = \{ X \in {}_H\mathcal{B} \mid K^{+} \subseteq \mathrm{Ann}(X) \} \\
    & = \{ X \in {}_H\mathcal{B} \mid a^l_X(i_K \otimes \id_X) = \varepsilon_H i_K \otimes \id_X \},
  \end{align*}
  where $i_K : K \to H$ is the inclusion morphism. From the last expression, we see that ${}_{H/J}\mathcal{B}$ contains the unit object of ${}_H\mathcal{B}$ and is closed under the tensor product of ${}_H\mathcal{B}$. Hence, by Lemmas~\ref{lem:braided-Hopf-ideals-1} and \ref{lem:braided-Hopf-ideals-2}, $J$ is a Hopf ideal of $H$.
\end{proof}

We establish a Galois connection as follows:

\begin{proposition}
  \label{prop:Galois-conn-CSA-Hopf-ideal}
  Let $\mathcal{B}$ be a braided tensor category, and let $H$ be a Hopf algebra in $\mathcal{B}$. Then there are well-defined order-preserving maps
  \begin{equation*}
    \begin{tikzcd}
      \{ \text{Hopf ideals of $H$} \}
      \arrow[r, shift left=1.2, "{\mathcal{F}}"]
      & \{ \text{normal left coideal subalgebras of $H$} \}
      \arrow[l, shift left=1.2, "{\mathcal{G}}"]
    \end{tikzcd}
  \end{equation*}
  given by $\mathcal{F}(J) = H^{\mathrm{co}(H/J)}$ and $\mathcal{G}(K) = H K^{+}$.
  Moreover, $\mathcal{F}$ and $\mathcal{G}$ form a monotone Galois connection: For a Hopf ideal $J$ and a normal left coideal subalgebra $K$ of $H$, we have
  \begin{equation}
    \label{eq:Galois-conn-CSA-Hopf-ideal}
    K \subseteq \mathcal{F}(J) \iff \mathcal{G}(K) \subseteq J.
  \end{equation}
\end{proposition}
\begin{proof}
  By Lemmas~\ref{lem:Hopf-ideal-to-CSA} and \ref{lem:CSA-to-Hopf-ideal}, the maps $\mathcal{F}$ and $\mathcal{G}$ are well-defined. It is obvious that the maps $\mathcal{F}$ and $\mathcal{G}$ preserve orders. It remains to show \eqref{eq:Galois-conn-CSA-Hopf-ideal}. Let $J$ be a Hopf ideal of $H$, and let $K$ be a left coideal subalgebra of $H$. We now assume that $K \subseteq \mathcal{F}(J)$ holds. Let $i_{K}^+ : K^+ \to H$ and $\pi : H \to H/J$ be the inclusion and the quotient morphism, respectively. Then we have
  \begin{equation*}
    \pi i_K^{+} = (\varepsilon_H \otimes \pi) \Delta_H i_K^{+}
    = (\varepsilon_H \otimes \pi u_H) i_K^{+}
    = \varepsilon_H i_K^{+} \otimes \pi u_H = 0,
  \end{equation*}
  where the assumption $K \subseteq \mathcal{F}(J)$ is used at the second equality. This implies that $J = \Ker(\pi)$ contains $K^{+}$. Since $J$ is an ideal of $H$, we have $\mathcal{G}(K) = H K^{+} \subseteq J$.

  To prove the converse, we assume $\mathcal{G}(K) \subseteq J$. Let $i_K : K \to H$ be the inclusion morphism. Since $K^{+} \subseteq J$ by the assumption, we have
  \begin{equation*}
    \pi i_K = \pi (i_K - u_H \varepsilon_H i_K) + \pi u_H \varepsilon_H i_K
    = \pi u_H \varepsilon_H i_K.
  \end{equation*}
  By definition, the left coaction $\delta_K : K \to H \otimes K$ of $H$ on $K$ is characterized by the equation $(\id_H \otimes i_K) \delta_K = \Delta_H i_K$. Thus we have
  \begin{gather*}
    (\id_H \otimes \pi)\Delta i_K
    = (\id_H \otimes \pi i_K) \delta_K
    = (\id_H \otimes \pi u_H \varepsilon_H i_K) \delta_K \\
    = (\id_H \otimes \pi u_H \varepsilon_H) \Delta_H i_K
    = (\id_H \otimes \pi u_H) i_K,
  \end{gather*}
  which implies $K \subseteq H^{\mathrm{co}(H/J)} = \mathcal{F}(J)$. The proof is done.
\end{proof}

\subsection{Proof of Theorem~\ref{thm:CSA-Hopf-ideal-bijection}}

Let $\mathcal{B}$ be a braided finite tensor category, and let $H$ be a Hopf algebra in $\mathcal{B}$. We take a nondegenerate braided finite tensor category $\widetilde{\mathcal{B}}$ such that $\mathcal{B} \subseteq \widetilde{\mathcal{B}}$ (for example, the Drinfeld center of $\mathcal{B}$). Then we have
\begin{equation*}
  \{ \text{Hopf ideals of $H \in \mathcal{B}$} \}
  = \{ \text{Hopf ideals of $H \in \widetilde{\mathcal{B}}$} \},
\end{equation*}
and the same holds for normal left coideal subalgebras of $H$. Thus, by replacing $\mathcal{B}$ with $\widetilde{\mathcal{B}}$, we may assume that $\mathcal{B}$ is nondegenerate. We note that, under this assumption, ${}^H_H\mathcal{YD}(\mathcal{B})$ is also a nondegenerate braided finite tensor category \cite[\S6]{shimizu2019non}.

\begin{proof}[Proof of Theorem~\ref{thm:CSA-Hopf-ideal-bijection}]
  For simplicity of notation, we set $\mathcal{C} = {}^H_H\mathcal{YD}(\mathcal{B})$.
  We consider the adjoint algebra $\mathbf{A}$ given in \S\ref{subsec:braided-Hopf-adj-alg}.
  Since it is constructed from a right adjoint of a central tensor functor, $\mathbf{A}$ is a simple commutative algebra in $\mathcal{C}$.
  As ${}^H_H\mathcal{YD}(\mathcal{B})$ is nondegenerate, Theorem~\ref{thm:nondegenerate-lattice-isomorphism} yields a bijection
  \begin{align*}
    \Phi : \Lalg{\mathbf{A}}
    & \to [\mathcal{C}_{\mathbf{A}}^{\mathrm{loc}},
      \mathcal{C}_{\mathbf{A}}]_{\otimes}, \\
    K & \mapsto \{ M \in \mathcal{C}_{\mathbf{A}} \mid M|_K \in \mathcal{C}_K^{\mathrm{loc}} \}.
  \end{align*}
  As we have remarked in \S\ref{subsec:braided-Hopf-adj-alg}, a subalgebra of $\mathbf{A}$ is the same thing as a normal left coideal subalgebra of $H$.
  We recall that there is an equivalence $\widetilde{R} : {}_H\mathcal{B} \to \mathcal{C}_{\mathbf{A}}$ of tensor categories. By Lemma~\ref{lem:local-condition-for-RVK} and the discussion in the proof of Lemma~\ref{lem:CSA-to-Hopf-ideal}, we have $\Phi(K) = \widetilde{R}({}_{H/HK^{+}}\mathcal{B})$ for all normal left coideal subalgebras $K$ of $H$. Thus, by composing the bijection $\Phi$, the bijection induced by the tensor equivalence $\widetilde{R}$ and the bijection given by Lemma~\ref{lem:Hopf-ideals-and-tensor-subcates}, we obtain a bijection
  \begin{equation*}
    \{ \text{normal left coideal subalgebras of $H$} \}
    \to \{ \text{Hopf ideals of $H$} \},
    \quad K \mapsto H K^{+}.
  \end{equation*}
  Namely, we have proved that the map $\mathcal{G}$ of Theorem~\ref{thm:CSA-Hopf-ideal-bijection} is bijective. Since $\mathcal{F}$ and $\mathcal{G}$ form a Galois connection (Proposition \ref{prop:Galois-conn-CSA-Hopf-ideal}), the map $\mathcal{F}$ is also bijective. The proof is done.
\end{proof}

\bibliographystyle{alpha}
\bibliography{references}

\end{document}